\documentclass[11pt]{ip-journal}

\usepackage{color}
\usepackage{url}
\usepackage[T1]{fontenc}
\usepackage{lmodern}
\usepackage{graphicx}
\usepackage{a4}
\usepackage{amssymb,latexsym}
\usepackage[mathscr]{eucal}
\usepackage{citesort}

\newcommand{\SSection}{Section}
\newcommand{\ptcheck}[1]{}

\newcommand{\erw}[1]{\mnote{{\bf erw:} #1}}

\newcommand{\zmcD}{\,\,\mathring{\!\!\mcD}}

\newcommand{\intOmega}{{\int_\Omega}}

\newcommand{\sources}{{\mcC}}
\newcommand{\thetaw}{w}
\newcommand{\hg}{{\hat g{}}}
\newcommand{\hgtwo}{g}
\newcommand{\lots}{\mathrm{l.o.t.}}

\newcommand{\wzg}{\widetilde{\zg}}
\newcommand{\definingz}{z}

\newcommand{\ringh}{\,\mathring{\! h}{}}
\newcommand{\nozHkabc}{\mathring H^k_{x \rho^{-1} ,\, x^a z^b \rho^c}(\Omega)}%
\newcommand{\Habc}[1]{\mathring H^{#1}_{x \rho^{-1},\,x^a z^b \rho^c}(\Omega)}%

\newcommand{\varphiepsiloneta}{\varphi_{\epsilon}}
\newcommand{\etaepsilon}{\epsilon}

\newcommand{\wG}{{\widetilde \Gamma}}
\newcommand{\wnabla}{{\widetilde \nabla}}

\newcommand{\divr }{\mbox{\rm div}\,}

\newcommand{\Lpsi}{L^2_{\psi}}

\newcommand{\Lpsione}{\zH^1_{\phi,\psi}}
\newcommand{\Lpsitwo}{\zH^2_{\phi,\psi}}

\newcommand{\Lpsikg}[2]{\zH^{#1}_{\phi,\psi}(#2)}

\newcommand{\zHkpp}{\zHk_{\phi,\psi}}
\newcommand{\Hkpp}{H^k_{\phi,\psi}}
\newcommand{\zHk}{\zH^k}
\newcommand{\zH}{\mathring{H}}

\newcommand{\hs}{\cH_{\mbox{\scriptsize sing}}}

\newcommand{\beadl}[1]{\begin{deqarr}\label{#1}}
\newcommand{\eeadl}[1]{\arrlabel{#1}\end{deqarr}}%

\def \Nat{\mathbb{N}}
\def\nz{\ifmmode {I\hskip -3pt N} \else {\hbox {$I\hskip -3pt N$}}\fi}
\def\zz{\ifmmode {Z\hskip -4.8pt Z} \else
       {\hbox {$Z\hskip -4.8pt Z$}}\fi}
\def\qz{\ifmmode {Q\hskip -5.0pt\vrule height6.0pt depth 0pt
       \hskip 6pt} \else {\hbox
       {$Q\hskip -5.0pt\vrule height6.0pt depth 0pt\hskip 6pt$}}\fi}
\def\rz{\ifmmode {I\hskip -3pt R} \else {\hbox {$I\hskip -3pt R$}}\fi}
\def\cz{\ifmmode {C\hskip -4.8pt\vrule height5.8pt\hskip 6.3pt} \else
       {\hbox {$C\hskip -4.8pt\vrule height5.8pt\hskip 6.3pt$}}\fi}
\def\au{{\setbox0=\hbox{\lower1.36775ex\hbox{''}\kern-.05em}\dp0=.36775ex\hs
kip0pt\box0}}
\def\ao{{}\kern-.10em\hbox{``}}

\newcommand\Gregbeq{\begin{eqnarray}}
\newcommand\Gregeeq{\end{eqnarray}}

\def\cH{{\cal H}}

\def\h1{{\hat 1}}
\def\h2{{\hat 2}}

\def\3f{\frac{3}{2}}

\newcommand{\roscoff}[1]{}

{\catcode `\@=11 \global\let\AddToReset=\@addtoreset}
\AddToReset{figure}{section}
\renewcommand{\thefigure}{\thesection.\arabic{figure}}

\DeclareFontFamily{OT1}{rsfs}{}
\DeclareFontShape{OT1}{rsfs}{m}{n}{ <-7> rsfs5 <7-10> rsfs7 <10-> rsfs10}{}
\DeclareMathAlphabet{\mycal}{OT1}{rsfs}{m}{n}

{\catcode `\@=11 \global\let\AddToReset=\@addtoreset}
\AddToReset{equation}{section}
\renewcommand{\theequation}{\thesection.\arabic{equation}}
\newcounter{mnotecount}[section]

\renewcommand{\themnotecount}{\thesection.\arabic{mnotecount}}

\newcommand{\wasdmug}{}

\newcommand{\oversetty}[2]{%
\mathop{#2}\limits^{\vbox to -.1ex{%
\kern -1.5ex\hbox{$\scriptstyle #1$}\vss}}}

\newcommand{\jlcax}[1]{}
\newcommand{\eean}{\nonumber\end{eqnarray}}

\newcommand{\ptcr}[1]{{\color{red}\mnote{{\color{red}{\bf ptc:}
#1} }}}

\newcommand{\kk}[1]{}

\newcommand{\mcH}{{\mycal H}}

\newcommand{\beq}{\begin{equation}}

\newcommand{\FS}       
                  {F}

\newcommand{\HS} 
       {H_{\mbox{\scriptsize volume}}}

{\ptc{this should be removed in the oberwolfach version}}%

\newcommand{\eeal}[1]{\label{#1}\end{eqnarray}}
\newcommand{\bed}{\begin{deqarr}}
\newcommand{\eed}{\end{deqarr}}
\newcommand{\bedl}[1]{\begin{deqarr}\label{#1}}
\newcommand{\eedl}[2]{\arrlabel{#1}\label{#2}\end{deqarr}}

\newcommand{\loc}{\textrm{\scriptsize\upshape loc}}

\newcommand{\tg}{{\widetilde{g}}}

\newcommand{\mcU}{{\mycal U}}

\newcommand{\mcK}{{\mycal K}}

\newcommand{\bel}[1]{\begin{equation}\label{#1}}
\newcommand{\bea}{\begin{eqnarray}}
\newcommand{\bean}{\begin{eqnarray}\nonumber}
\newcommand{\beal}[1]{\begin{eqnarray}\label{#1}}
\newcommand{\eea}{\end{eqnarray}}

\newcommand{\nn}{\nonumber}
\newcommand{\Eq}[1]{Equation~\eq{#1}}

\newcommand{\myqed}{{\hfill $\Box$}}
\newcommand{\qedskip}{\endproof}
\newcommand{\be}{\begin{equation}}
\newcommand{\eeq}{\end{equation}}
\newcommand{\ee}{\end{equation}}
\newcommand{\beqa}{\begin{eqnarray}}
\newcommand{\eeqa}{\end{eqnarray}}
\newcommand{\beqan}{\begin{eqnarray*}}
\newcommand{\eeqan}{\end{eqnarray*}}
\newcommand{\ba}{\begin{array}}
\newcommand{\ea}{\end{array}}

\newcommand{\mcD}{{\mycal D}}

\newcommand{\mcV}{{\mycal V}}

\newcommand{\mnote}[1]
{\protect{\stepcounter{mnotecount}}$^{\mbox{\footnotesize
$
\bullet$\themnotecount}}$ \marginpar{
\raggedright\tiny\em
$\!\!\!\!\!\!\,\bullet$\themnotecount: #1} }

\newcommand{\warn}[1]
{\protect{\stepcounter{mnotecount}}$^{\mbox{\footnotesize
$
\bullet$\themnotecount}}$ \marginpar{
\raggedright\tiny\em
$\!\!\!\!\!\!\,\bullet$\themnotecount: {\bf Warning:} #1} }

\newcommand{\Ricc}{\mathrm{Ric}\,}

\newcommand{\R}{\mathbb R}

\newcommand{\N}{\mathbb N}
\newcommand{\Z}{\mathbb Z}

\newcommand{\bM}{\overline{\! M}}

\newcommand{\eq}[1]{(\ref{#1})}

\newcommand{\ptc}[1]{\mnote{{\bf ptc:}#1}}

\newcommand{\Ric}{\mbox{\rm Ric}}

\newcommand{\mcC}{{\mycal C}}

\newcommand{\beqar}{\begin{deqarr}}
\newcommand{\eeqar}{\end{deqarr}}

\newcommand{\beaa}{\begin{eqnarray*}}
\newcommand{\eeaa}{\end{eqnarray*}}

\newcommand{\tr}{\mathrm{tr}}
\newcommand{\zg}{\mathring{g}}

\newcommand{\znabla}{\mathring{\nabla}}

\newcommand{\bethm}{\begin{theorem}}
\newcommand{\et}{\end{theorem}}
\newcommand{\bl}{\begin{Lemma}}

\newtheorem{Theorem} {\sc  Theorem\rm} [section]
\newtheorem{theorem} [Theorem] {\sc  Theorem\rm}

\newtheorem{corollary} [Theorem] {\sc  Corollary\rm}

\newtheorem{cor} [Theorem] {\sc  Corollary\rm}
\newtheorem{Lemma} [Theorem] {\sc  Lemma\rm}
\newtheorem{Proposition} [Theorem] {\sc  Proposition\rm}
\newtheorem{prop} [Theorem] {\sc  Proposition\rm}

\newtheorem{Definition}[Theorem]{\sc  Definition\rm}

\newtheorem{lemma} [Theorem] {\sc  Lemma\rm}

\newtheorem{proposition} [Theorem] {\sc  Proposition\rm}

\newtheorem{Remark}[Theorem]{\sc Remark\rm}

\newtheorem{remark}[Theorem]{\sc Remark\rm}

\theoremstyle{nonumberplain}

\DeclareFontFamily{OT1}{rsfs}{}
\DeclareFontShape{OT1}{rsfs}{m}{n}{ <-7> rsfs5 <7-10> rsfs7 <10-> rsfs10}{}
\DeclareMathAlphabet{\mycal}{OT1}{rsfs}{m}{n}

{\catcode `\@=11 \global\let\AddToReset=\@addtoreset}
\AddToReset{figure}{section}
\renewcommand{\thefigure}{\thesection.\arabic{figure}}

{\catcode `\@=11 \global\let\AddToReset=\@addtoreset}
\AddToReset{equation}{section}
\renewcommand{\theequation}{\thesection.\arabic{equation}}

\renewcommand{\ptc}[1]{}
\renewcommand{\ptcr}[1]{}
\renewcommand{\erw}[1]{}

\begin{document}

\title{Exotic hyperbolic gluings}

\author{Piotr T.~Chru\'sciel}
\address{
Faculty of Physics and Erwin Schr\"odinger Institute\\
 University of Vienna \\Boltzmanngasse 5\\ A 1090 Wien, Austria\\}
\email{\url{piotr.chrusciel@univie.ac.at}\\
URL \url{http://homepage.univie.ac.at/piotr.chrusciel}}

\author{Erwann Delay}
\address{
Laboratoire de Math\'ematiques d'Avignon\\
UFR-ip Sciences, Technologies, Sant\'e\\
Campus Jean-Henri Fabre\\
301 rue Baruch de Spinoza\\
BP 21239,
F-84916 Avignon Cedex 9,
France\\}
\email{
\url{erwann.delay@univ-avignon.fr}\\
URL \protect\url{http://math.univ-avignon.fr}}

\maketitle
\begin{abstract}
We carry out ``exotic gluings'' a la Carlotto-Schoen for asymptotically hyperbolic general relativistic initial data sets. In particular we obtain a direct construction of non-trivial initial data sets which are exactly hyperbolic in large regions extending to conformal infinity.
\end{abstract}

\tableofcontents

\section{Introduction}
 \label{s23III13.1}

In an outstanding paper~\cite{CarlottoSchoen}, Carlotto and Schoen have shown that gravity can be screened away, using a gluing construction which produces asymptotically flat solutions of the general relativistic constraint equations which are Minkowskian within a solid cone. The object of this work is to establish a similar result for asymptotically hyperbolic initial data sets in all dimensions $n\ge 3$.

Our result has a direct analogue in a purely Riemannian setting of asymptotically hyperbolic metrics with constant scalar curvature; this corresponds to vacuum general relativistic initial data sets where the extrinsic curvature tensor is pure trace.
We present a simple version of the gluing here, the reader is referred Section~\ref{sec:def} for precise definitions, to Theorems~\ref{T21IX15.1}, \ref{propRcassimple2} and \ref{Thefulltheorem}  for more general results and to Theorem~\ref{T3X15.1} for an application.
 \ptcr{added}

  Consider a  manifold $M$ with two asymptotically hyperbolic metrics $g$ and $\hat g$.
 Assume that $g$ and $\hat g$ approach the same hyperbolic metric $h$ as  the conformal boundary at infinity is approached. We use the half-space-model coordinates near the conformal boundary, so that
$$
 \mcH
  =\{(\theta ,z)|\ z > 0, \theta  \in\R^{n-1}\}\subset \R^n
  \,,
  \qquad
   h = z^{-2} (dz^2 + d\theta^2)
   \,.
$$
We set\erw{  delta squared}
\bel{18V15.2}
 B_\delta  := \{z>0\,,\    |\theta |^2 + z^2 <\delta^2
 \}
 \,,
 \quad
 \complement B_\delta := M\setminus B_\delta
 \,.
\ee
We use the above coordinates as local coordinates near a point at the conformal boundary for asymptotically hyperbolic metrics.

Let $\epsilon$ be a small scaling parameter. A special case  of Theorem~\ref{propRcassimple2} below reads:

\begin{Theorem}
  \label{T1815.2}
Let $k>n/2$ and let  ${\hg}$, ${\hgtwo}$ be  $C^{k+4}$-asymptotically
hyperbolic  metrics on an $n$-dimensional manifold $M$.
There exists $0<\epsilon_0<1$ such that for all $0<\epsilon<\epsilon_0$ there exists an asymptotically hyperbolic metric $g_\epsilon$, of $C^{k+2-\lfloor n/2\rfloor}$
differentiability class, and with  scalar curvature lying between the scalar curvatures of $\hg$ and $\hgtwo$ such that
\bel{18V15.3}
 g_\epsilon|_{B_{\epsilon   }} = {\hg}
 \,,
 \quad
 g_\epsilon|_{\complement B_{2\epsilon }} = {\hgtwo}
  \,.
\ee
\end{Theorem}

Note that if both metrics have the same constant scalar curvature, then so will the final metric.

When $g$ and $\hat g$  have a well defined hyperbolic mass, then so does $g_\epsilon$, with the mass of $g_\epsilon$ tending to that of $g$ as $\epsilon$ tends to zero. The reader is referred to Remark~\ref{R21XI15.1} below for a more detailed discussion.

A case of particular interest arises when $g$ has constant scalar curvature and $\hat g$
is the hyperbolic metric. Then the construction above provides a constant scalar curvature metric which coincides with the hyperbolic metric on an open region $B_\epsilon$ extending all the way to the conformal boundary.
This, and variations thereof discussed below, provides in particular examples of large families of CMC horospheres,  hyperbolic\erw{changed flat to hyperbolic} hyperplanes, etc.\ (compare~\cite{MazzeoPacardCMCinAH}), in families of asymptotically hyperbolic constant scalar curvature metrics which are \emph{not} exactly hyperbolic.

As already mentioned, our gluing results apply to asymptotically hyperbolic general relativistic data sets, the reader is referred to Section~\ref{ss18IX15.7} for precise statements.

The differentiability requirements for the initial metrics, as well as the regularity of the final metric, can be improved using techniques as in~\cite{ChDelayHilbert,CCI2},
but we have not attempted to optimize the result.

While the strategy of the proof is known in principle~\cite{ChDelay}, its implementation requires considerable analytical effort.
At the heart of the proof lie the ``triply weighted'' Poincar\'e inequality of Theorem~\ref{T21XI15.1} and the ``triply weighted'' Korn inequality of Theorem~\ref{TC22V15} below.

 \ptcr{additions to the end of the section}

This paper is organised as follows: In Section~\ref{sec:def} the definitions are presented, and  notation  and conventions are spelled out. Our gluing results are presented and proved in Section~\ref{s18IX15.1}, modulo some key technicalities which are deferred to Section~\ref{s18IX15.6} (where the required weighted Poincar\'e inequalities are established) and Section~\ref{s7VII14.1} (where the required weighted Korn inequalities are proved), as well as Appendix \ref{s2X15.1} (where nonexistence of KIDs satisfying the required boundary conditions is established).

The reader is invited to consult~\cite{Corvino,CorvinoSchoen,CorvinoHuang,CCI2} for the original papers as well as further applications of  gluing constructions in general relativity.

\section{Definitions, notations and conventions}
 \label{sec:def}

We use the summation convention throughout, indices are lowered with $g_{ij}$
and raised with  its inverse $g^{ij}$.

We will have to frequently control the asymptotic behavior of the objects at hand. Given a tensor field $T$ and a function $f$,  we will write
$$
T=O_g(f),
$$
when there exists a constant $C$ such that the $g$-norm of $T$ is dominated by $Cf$.

\subsection{Asymptotically hyperbolic manifolds}
 \label{ss20XI15.13}

Let $\overline{M}$ be a smooth  $n$-dimensional
manifold with boundary $\partial {M}$. Thus
$M:=\overline{M}\backslash\partial{M}$ is a   manifold
without boundary.
(We use the
analysts' convention that a manifold $M$ is always open; thus a
manifold $M$ with non-empty boundary $\partial M$ does not contain
its boundary; instead, $\bM= M \cup
\partial M$ is a manifold with boundary in the differential
geometric sense.) Unless explicitly specified otherwise \emph{no}
conditions on $M$ are made  --- \emph{e.g.}\/ that $\partial M$ or $\overline M$
are compact, except that $M$ is a smooth manifold; similarly no conditions
\emph{e.g.}\/ on completeness of $(M,g)$, or on its radius of
injectivity, are made.

The boundary $\partial {M}$
will play the role of a \emph{conformal boundary at infinity}
of $M$.  Throughout the symbol $\definingz $  will denote a defining function for
$\partial M$, that is a non-negative smooth function on
$\overline{M}$, vanishing precisely on $\partial M$, with
$d\definingz $ never vanishing  there.

A metric $g$ on $M$ will be called \emph{asymptotically hyperbolic}, or AH,  if there exists a smooth
defining function $\definingz $ such that $\tilde g=z^2 g$ extends by continuity to a metric on $\overline M$,
with $|d\definingz|^2_{\tilde g}=1+O(z)$.
The terminology is motivated by the fact that, under rather weak differentiability hypotheses, the sectional
curvatures of $g$ tend to $-1$ as $\definingz $ approaches zero; cf., e.g.,
\cite{Mazzeo:hodge}. We will typically assume more differentiability of $\tilde g$, as will be made precise when need arises. In particular $\tilde g$ will be assumed to be differentiable up-to-boundary.

It is well known that, near infinity, for any sufficiently
differentiable asymptotically hyperbolic metric $g$ we may choose the
defining function $\definingz $ to be the $\tg $-distance to the
boundary, and that there exist local coordinates $(z,\theta^A)$ near $\partial M$ so that  on $(0,\epsilon)\times\partial M$ the
metric takes the form
\bel{metrichgold}
 g=
  \definingz ^{-2}(d\definingz ^2+{\widetilde g}_{AB}(\definingz, \theta^C )d\theta^Ad\theta^B) \,,
\ee
where $z$ runs over the first factor of $(0,\epsilon)\times\partial M$, $\theta^A$ are local coordinates on $
\partial M$, and where $\{\widetilde{g}_{AB}(\definingz, \cdot)d\theta^Ad\theta^B\}_{\definingz \in[0,\epsilon]}$ is a family of
uniformly equivalent, metrics on $\partial M$.

\begin{Definition}
   \label{D21IX15.1} Let $\zg$ be an asymptotically hyperbolic metric such that $\wzg$ is smooth on $\overline M$ and
let $W$ be a function space.
An asymptotically hyperbolic metric $g$ will be said to be of $M_{\zg+W} $ class if  $g-\zg\in W$.
\end{Definition}

We will be typically interested in $M_{\zg+W}$ metrics with $W=C^k_{1,z^{-\sigma}}$ with $\sigma>0$, see Section~\ref{SwSs}
 for notation. In local coordinates as above such metrics decay to the model metric as $z^{\sigma-2}$, or as $z^{\sigma}$ in $g$-norm,
 are continuously compactifiable, with derivatives satisfying uniform weighted estimates near the boundary. Further, there exists then a constant $C$ such that
\bel{gestimb}
  |g-\zg|_{\zg}+|\znabla g|_{\zg}+...+|\znabla^{(k)}g|_{\zg}\leq
 C\definingz ^\sigma\,,
\ee
where the norm and covariant derivatives $\znabla$ are defined by $\zg$.
 For $\R\ni\sigma>k\ge 1$, $k\in \N$ and $g\in M_{\zg+C^k_{1,z^{-\sigma}}}$ the conformally  rescaled metrics   $z^2 g$  can be extended  to the conformal boundary of $M$, with the extension belonging to the  $C^{\lfloor\sigma\rfloor}$-differentiability class.
%
%
%

\subsection{Weighted Sobolev and  weighted H\"older spaces}\label{SwSs}

Let $\phi$ and $\psi$ be two smooth strictly
positive functions on
$M$. The function $\psi$ is used to control the growth of the fields involved near boundaries or in the asymptotic regions, while $\phi$ allows the growth to be affected by derivation.
 \ptcr{added}
For $k\in \Nat$ let $\Hkpp(g) $ be the space of $H^k_\loc$
functions or tensor fields such that the norm\footnote{The reader
is referred to~\cite{Aubin,Aubin76,Hebey} for a discussion of
Sobolev spaces on Riemannian manifolds.}
\be \label{defHn}
 \|u\|_{\Hkpp (g)}:=
 (\int_M(\sum_{i=0}^k \phi^{2i}|\nabla^{(i)}
 u|^2_g)\psi^2 d\mu_g )^{\frac{1}{2}}
\ee
is finite, where
$\nabla^{(i)}$ stands for the tensor $\underbrace{\nabla ...\nabla
}_{i \mbox{ \scriptsize times}}u$, with $\nabla$ --- the
Levi-Civita covariant derivative of $g$; we assume throughout that
the metric is at least $W^{1,\infty}_\loc$; higher
differentiability will be usually indicated whenever needed.

For
$k\in \Nat$ we denote by $\zHkpp $ the closure in $\Hkpp$ of the
space of $H^k$ functions or tensors which are compactly (up to a
negligible set) supported in $M$, with the norm induced from
$\Hkpp$.
The $\zHkpp $'s are Hilbert spaces with the obvious scalar product
associated to the norm \eq{defHn}. We will also use the following
notation
$$
\quad \zHk  :=\zHk  _{1,1}\,,\quad
L^2_{\psi}:=\zH^0_{1,\psi}=H^0_{1,\psi}\,,
$$ so that $L^2\equiv \zH^0:=\zH^0_{1,1}$.

For $\phi$ and $\varphi$  --- smooth strictly positive functions
on M, and for $k\in\N$ and $\alpha\in [0,1]$, we define
$C^{k,\alpha}_{\phi,\varphi}$ the space of $C^{k,\alpha}$
functions or tensor fields  for which the norm
$$
\begin{array}{l}
\|u\|_{C^{k,\alpha}_{\phi,\varphi}(g)}=\sup_{x\in
M}\sum_{i=0}^k\Big(
\|\varphi \phi^i \nabla^{(i)}u(x)\|_g\\
 \hspace{3cm}+\sup_{0\ne d_g(x,y)\le \phi(x)/2}\varphi(x) \phi^{i+\alpha}(x)\frac{\|
\nabla^{(i)}u(x)-\nabla^{(i)}u(y)\|_g}{d^\alpha_g(x,y)}\Big)
\end{array}
$$ is finite.

\section{Asymptotically hyperbolic metrics and initial data sets}
 \label{s18IX15.1}

In this section we will construct metrics with ``interpolating scalar curvature'', in the sense of \eq{21IX15.1} below,  in three geometric setups.

As already pointed out, the symbol $z$  denotes a defining function for $\partial M$ as in \eq{metrichgold}. In particular $z$ will be smooth on $\overline M$, with $z>0$ on $M$, vanishing precisely on $\partial M$.

We will be gluing together metrics which are close to each other on a set $\Omega\subset M$. We will use the symbol $x$ to denote a defining function for $ \overline{\partial \Omega\cap M}$ (closure in $\overline M$), thus $x$ is smooth-up-to boundary on $\overline \Omega$,   $x>0$ on $\Omega$, with $x=0$ precisely on $\overline{\partial \Omega\cap M}$, and with $dx$ nowhere zero on $\overline{\partial \Omega\cap M}$.
(The reader is warned that in our applications the boundary of $\Omega$ in $M$ and the boundary of $\Omega$ in $\overline M$ do not coincide; as an example, see  the set $\Omega$ of \eq{29II16.1} below.)

Throughout this section we use weighted functions spaces with weights
\bel{21IX15.5}
\phi=\frac x\rho\,,\;\;\;\psi= x^az^b\rho^c
\,,
\
\mbox{where} \
  \rho=\sqrt{x^2+z^2}
 \,.
\ee

A rather simple situation occurs when $\Omega$ is a half-annulus centered at the conformal boundary, this is considered in Section~\ref{ss18IX15.1}.

More generally, we consider sets $\Omega$ such that the closure $\overline{\partial \Omega \cap M}$ in $\overline M$ of $\partial \Omega\cap M$ is smooth and compact in the conformally compactified manifold, with two connected components.
This is described in
 \SSection~\ref{ss18IX15.2}.

Finally,  we can glue-in an exactly hyperbolic region to any asymptotically hyperbolic metric in a half-ball near the conformal boundary. This is described in  \SSection~\ref{ss18IX15.3}.

In \SSection~\ref{ss18IX15.7} we turn our attention to initial data for vacuum Einstein equations. We show there how to extend the proofs to  such data.

In  \SSection~\ref{ss18IX15.4} we use the results of   \SSection~\ref{ss18IX15.3} to show how to make a Maskit-type gluing of two asymptotically hyperbolic manifolds.

\erw{added}Throughout this
section, we let $\mathring g$
be any fixed background asymptotically hyperbolic metric on M, as explained in
\SSection~\ref{ss20XI15.13}.

\subsection{A half-annulus with nearby metrics}
 \label{ss18IX15.1}

While our gluing construction will apply to considerably more general situations, in this
section we describe a simple setup  of interest. We choose the underlying manifold to be the ``half-space model'':
$$
 \mcH
  =\{(z,\theta)|\ z>  0, \theta \in\R^{n-1}\}\subset \R^n
  \,,
  \quad
 \overline{ \mcH}
  =\{(z,\theta)|\ z\ge  0, \theta \in\R^{n-1}\}
  \,.
$$
We will glue together metrics asymptotic to each other while interpolating their respective
scalar curvatures.
The first metric will be assumed to be of $M^{k+4}_{\zg+C^1_{1,z}}$-differentiability class and therefore, in suitable local coordinates, will take the form
\bel{bgenerale}
 g =\frac {1} {z^2}
  \big(
   {(1+O(z))dz^2 + \underbrace{h_{AB}(z,\theta^C)d\theta^A d\theta^B}_{=:h(z)}}
    +O(z)_Adz\, d\theta^A
   \big)
 \,,
\ee
where $h(z)$ is a continuous family of Riemannian metrics on $\R^{n-1}$.

We define \erw{typo corrected}
$$
 B_\lambda  := \{z>0\,,\    \underbrace{\sum_i(\theta^i)^2}_{=:|\theta|^2} + z^2 <\lambda^2\}
 \,,\;\;\;A_{\epsilon,\lambda}=B_{\lambda}\setminus \overline{B_\epsilon}
 \,.
$$
The gluing construction will  take place in the region
\bel{29II16.1}
 \Omega=A_{1,4}
 \,.
\ee
We take $x$ to be any smooth function on $ \Omega$ which equals the $z^2\mathring g$-distance
to $\{|\theta|^2 + z^2=1\}\cup \{|\theta|^2 + z^2=4 \}$ near this last set.

In fact  we only need the  metric $ g$ to be defined on, say, $B_5$.

Let $\hg$ be a second metric   on $B_5$ which is close to $ g$ in $C^{k+4}_{1,z^{-\sigma}}(A_{1,4})$.
Let $\chi$ be a smooth non-negative function on  $\mcH$, equal to $1$  on $\mcH\setminus B_3$, equal to zero
on $B_2$, and positive on $\mcH\backslash \overline{B_2}$.
Following~\cite{Erwanninterpolating}, let
\bel{21IX15.1}
g_\chi=\chi \hg+(1-\chi) g\,,
\qquad
R_\chi=\chi R(\hg)+(1-\chi)R( g)
 \,.
\ee
Our result in this context is a special case of Theorem~\ref{T21IX15.1} in the next section (see also the remarks after the theorem there):

\begin{theorem}
 \label{propRcassimple}
Let $n/2<k<\infty$, $b\in[0,\frac{n+1}2]$, $\sigma>\frac{n-1}2+b$, suppose that $g\in M_{\zg+C^{k+4}_{1,z^{-1}}}$.
For all 
$\hg$ close enough to $ g$ in $C^{k+4}_{1,z^{-\sigma}}(A_{1,4})$
there exists a
two-covariant symmetric tensor field $h$ in $C^{k+2-\lfloor n/2\rfloor}(\mcH)$, vanishing outside of $A_{1,4}$, such that the tensor field $g_\chi+h$ defines a metric satisfying
\bel{solmodker}
 R(g_\chi+h)=R_\chi
 \,.
\ee
\end{theorem}

\begin{Remark}
{\rm
The construction invokes weighted Sobolev spaces on $A_{1,4}$ with $\phi =x/\rho$, $\psi =x^a z^b \rho^c$, where $a$ and $c$ are chosen large as determined by $k$, $n$ and $\sigma$.
The tensor field $h$ given by Theorem~\ref{propRcassimple} satisfies
$$h
 \in \psi^2\phi^2\mathring H^{k+2}_{\phi,\psi}(g_\chi)
 \,,
$$
and  there exists a constant $C$ independent of $\hat g$ such that
\bel{21IX15.11}
 \|h\|_{\psi^2\phi^2\mathring H^{k+2}_{\phi,\psi}(g_\chi)} \le C
    \|R(\hat g) - R(g)
        \|_{ \mathring H^{k}_{1,z^{-b}}(g_\chi)}
 \,.
\ee
The reader is referred to Remark~\ref{R19IX15.1} below for a description of the behaviour of $h$ near the corner $\rho=0$.
}
\end{Remark}

\subsection{A fixed region with nearby metrics}
 \label{ss18IX15.2}

As already pointed out, Theorem~\ref{propRcassimple} is a special case of Theorem~\ref{T21IX15.1} below, which applies to the following setup:

Consider an asymptotically hyperbolic manifold $(M,g)$. Let    $\Omega\subset M$  be a  domain which is relatively compact in $\overline M$ and such that the boundary of $ \Omega$ in $M$
 is the union of two smooth hypersurfaces $\Sigma$ and $\hat \Sigma$. Both $\Sigma$ and $\hat \Sigma$ are assumed to meet  the conformal boundary $\partial M$ smoothly and transversally, with $\overline{\Sigma}\cap \overline{\hat \Sigma} = \emptyset$ (closure in $\overline M$). Furthermore, we require that $\Omega$ lies to one side of each $\Sigma$ and $\hat \Sigma$. (In the setup of the previous section we have $ M= \mcH$, $\Omega =  A_{1,4}$, with  $\Sigma$ and $\hat \Sigma$ being the open half-spheres forming the connected components of $\partial A_{1,4}\cap \mcH$.)

As already pointed out, we denote by $x$ a smooth-up-to-boundary defining function for $\overline{\partial \Omega\cap M}$, strictly positive on $\Omega$.

Recall that the linearized scalar curvature operator $P=P_g$ is
$$
P_g h := DR(g)h=-\nabla^k\nabla_k(tr_g
\;h)+\nabla^k\nabla^lh_{kl}-R^{kl}h_{kl}
 \,,
$$
so that its $L^2$ formal adjoint reads
\bel{Pstar}
\begin{array}{lllll}
    P^*_gf&=&[DR(g)]^*f&=&-\nabla^k\nabla_k fg+\nabla\nabla
    f-f\,\Ric(g)
 \,.
\end{array}
\ee

We consider two asymptotically hyperbolic metrics $g$ and $\hg$ defined on $\Omega$.
Let $\chi$ be a smooth non-negative function on $\Omega $ which is one near $\hat\Sigma$ and is zero near $  \Sigma$.
Assuming that the metrics $g$ and $\hg$ are close enough to each other on $\Omega $ in a $z$-weighted norm, we can glue them together with interpolating curvature as in \eq{21IX15.1}:

\begin{theorem}
 \label{T21IX15.1}
Let $n/2<k<\infty$, $b\in[0,\frac{n+1}2]$, $\sigma>\frac{n-1}2+b$,  $\phi= x/\rho$, $\psi = x^a z^b \rho^c$, suppose that $g\in M_{\zg+C^{k+4}_{1,z^{-1}}}$.
For all real numbers $a$ and $c$ large enough and for all
$\hg$ close enough to $ g$ in $C^{k+4}_{1,z^{-\sigma}}(\Omega )$
there exists a unique two-covariant symmetric tensor field of the form
$$
 h=\psi^2\phi^{4} P_{g_\chi}^*u
 \in \psi^2\phi^2\mathring H^{k+2}_{\phi,\psi}(g_\chi)= \mathring H^{k+2}_{\phi,\psi^{-1}\phi^{-2}}(g_\chi)
$$
such that  $g_\chi+h$ defines a metric satisfying
\bel{solmodker2}
 R(g_\chi+h)=R_\chi
 \,.
\ee
Moreover there exists a constant $C$ such that
\bel{21IX15.13}
 \|h\|_{\psi^2\phi^2\mathring H^{k+2}_{\phi,\psi}(g_\chi)} \le C
    \|R(\hat g) - R(g)
        \|_{ \mathring H^{k}_{1,z^{-b}}(g_\chi)}
 \,.
\ee
The tensor field $h$ vanishes at $\partial \Omega$ and can be $C^{k+2-\lfloor n/2\rfloor}$-extended by zero across $\partial \Omega$.
\end{theorem}

\begin{remark}{\rm
 \label{R19IX15.1}
Using a weighted Sobolev embedding we find
$$
 h\in \psi^2\phi^2\mathring H^{k+2}_{\phi,\psi}(g_\chi)
 \subset x^{2a+2} z^{2b} \rho^{2c-2} C^{k+2-\lfloor n/2\rfloor + \alpha}_{x^{ a+n/2} z^{ b}\rho^{ c - n/2}}
=  C^{k+2-\lfloor n/2\rfloor + \alpha}_{x^{-a+n/2-2} z^{-b}\rho^{-c -n/2+2}}
$$
where $\alpha$ is any number in $(0,1)$ when $n$ is even, and $\alpha=1/2$ when $n$ is odd, and in
fact it holds   (see Appendix~\ref{A12VII15.1})
$$
 h=o_g(x^{a-n/2+2} z^{b } \rho^{c+n/2-2})
 \,.
$$
We also note that $ \|R(\hat g) - R(g)
        \|_{ \mathring H^{k}_{1,z^{-b}} }\le C
    \|R(\hat g) - R(g)
        \|_{ \mathring C^{k}_{1,z^{-\sigma}} }$ for $\sigma>\frac{n-1}2+b$, hence we also have
\bel{21IX15.12}
 \|h\|_{\psi^2\phi^2\mathring H^{k+2}_{\phi,\psi}(g_\chi)}  \le C^2
    \|R(\hat g) - R(g)
        \|_{ \mathring C^{k}_{1,z^{-\sigma}}(g_\chi)}
 \,.
\ee
We finally note that $\psi^2\phi^2\mathring H^{k+2}_{\phi,\psi}(g_\chi)= \mathring H^{k+2}_{\phi,\psi^{-1}\phi^{-2}}(g_\chi)$, as follows from \eq{lcond} below (see also \cite{ChDelay} Appendix A).
\myqed
}
\end{remark}

\begin{remark}{\rm
 \label{R21XI15.2}
The cutoff function $\chi$ used to interpolate the scalar
curvatures can be chosen to be different from the one
interpolating the metrics.
\myqed
}
\end{remark}

\begin{remark}
 \label{R21XI15.1}
{\rm
Suppose that $\mathring g$ is an asymptotically hyperbolic metric as in~\cite{ChHerzlich,Wang}. The total energy-momentum vectors~\cite{ChHerzlich,Wang} with respect to $\zg$ of the metrics $g$ and $\hat g$ will be well defined if they both approach $\mathring g$ as $o_g(z^{n/2})$, with scalar curvatures approaching that of $\zg$ suitably fast.
The final metric will asymptote to $\zg$ as $o_g(z^{n/2})$, and therefore will have a well defined energy-momentum, if we require that $b\ge \frac n 2$; such a choice forces  $\sigma> n -\frac 12$. Note that the decay  towards $\zg$ of the metrics constructed above might be slower in the gluing region than   that of time-symmetric slices in Kottler-Schwarzschild-anti de Sitter metrics.

In our construction the fastest decay of the perturbation of the metric is obtained by setting $b=\frac{n+1}2$, this forces $\sigma>n$, in which case both $g$ and $\hat g$ must have the same energy-momentum. But a choice of $b\in [ n/2, (n+1)/2)$ allows gluing of metrics with different energy-momenta, and nearby seed metrics will lead to a glued metric with nearby total energy-momentum.
\myqed
}
\end{remark}

{\sc \noindent Proof of Theorem~\ref{T21IX15.1}:}
We want to apply~\cite[Theorem~3.7]{ChDelay}  with $K=Y=J=0$  and $\delta\rho=R_\chi-R(g_\chi)$, to solve
the equation
\bel{3X15.1}
 \underbrace{\psi^{-2}[R(g_\chi+h)-R(g_\chi)]}_{=:
 r_\chi(h)}
 =\psi^{-2} [R_\chi-R(g_\chi)]
\,,
\ee
with $g_\chi$  close to $ g$. This requires verifying the hypotheses thereof.

We start by noting  that $R_\chi-R(g_\chi)$ vanishes
near the  boundary $\{x=0\}$, is  $O(z^\sigma)$ near the boundary at infinity $\{z=0\}$, and tends to zero together with $g_\chi- g$ when
$\hg$ approaches $ g$.
The condition on $\sigma$ implies that
$$\psi^{-2} [R_\chi-R(g_\chi)]\in \mathring H^k_{\phi,\psi}
 \,.
$$

Let us verify that the weight functions \eq{21IX15.5} satisfy
the conditions (A.2) (namely \eq{lcond} below), as well as (B.1) and (B.2) of~\cite{ChDelay}.
For simplicity of the calculations below it is convenient to assume that
  \bel{26V15.1}
  \mbox{$\partial \Omega$ is $\tg$-orthogonal to $\partial M$}
  \quad
  \Longleftrightarrow
  \quad
   g(\nabla x, \nabla z)= O(z^3)
   \,.
  \ee
If this does not hold, we choose $\Omega'\subset \Omega$ for which \eq{26V15.1} holds, and we continue the construction with $\Omega$ replaced by $\Omega'$. The set $\Omega'$ can be chosen so that the cut-off function $\chi$ remains constant near the boundaries of $\Omega'$.
The tensor field  $h$ constructed on $\Omega'$ will also provide a tensor field which has the desired properties on the original $\Omega$.
 \ptcr{reworded}

Without loss of generality we can choose the defining function  $x:\overline \Omega\to\R^+$ of $\overline{\partial \Omega\cap M }$
 so that
  \bel{25V15.2}
   |dx|_\tg = 1 + O( x )
   \,.
  \ee
A calculation gives (see \eq{tracenablauV} below with $V=\nabla u=\nabla\ln\psi$)
	\bean
	\psi^{-2}\phi^2|\nabla\psi|^2&=&\frac{x^2}{\rho^2}\bigg\{
     a^2\frac{z^2}{x^2}+b^2+[c^2+ 2ca + 2bc ]\frac{z^2}{\rho^2}
\nonumber
\\
    &&
     + O(z) + a \, O(\frac{z^2}x)+2a\,b \, O(\frac{z^2}x)\bigg\}
     \,.
		\eean
Expanding the right-hand side and using
	\bel{inegazyrho}
	0 <\frac {x}\rho\leq 1\,,\qquad
0<  \frac z\rho \leq 1
\,,
	\ee
we  see that $\psi^{-2}\phi^2|\nabla\psi|^2$ is bounded. The same formula with $a=1$, $b=0$ and $c={-1}$   similarly shows that
 $|\nabla\phi|^2$
	is  bounded.

Induction, and similar calculations establish the higher-derivatives inequalities  in
\begin{equation}\label{lcond}
 |\phi^{i-1}\nabla^{(i)}\phi|_g\leq C_{i}\,,\;\;\;
|\phi^{i}\psi^{-1}\nabla^{(i)}\psi|_g\leq C_{i}\,.
\end{equation}

Condition~(B.1) of~\cite{ChDelay} is clearly verified.

For condition~(B.2)  of~\cite{ChDelay}, we recall that if $p=(x,y,z)$ is in a $\mathring g$-ball
	 centred at $p_0=(x_0, y_0,z_0)$ and of radius $\epsilon\phi(p_0)$ then the $z^2\mathring g$-distance  from $p$ to $p_0$
	is bounded up to a multiplicative constant by $\epsilon z_0\phi(p)=\epsilon {z_0x_0}/{\rho_0}$.
	In particular from (\ref{inegazyrho})
	$$
	|z-z_0|\leq \epsilon\frac{z_0x_0}{\rho_0}\leq \epsilon z_0,
	$$
	and
	$$
	|x-x_0|\leq \epsilon\frac{z_0x_0}{\rho_0}\leq \epsilon x_0.
	$$
	This proves that, on this ball, $z$ is equivalent to $z_0$ and $x$ equivalent to $x_0$, thus
	$x/\rho$ is equivalent to $x_0/\rho_0$, which  is exactly  condition~(B.2) of~\cite{ChDelay}.

\medskip

To continue, recall that elements of the kernel of the linearized scalar-curvature map are called \emph{static KIDs}.
We need to check that within our range of weights
\bea
 \label{-+23VI14.2ah}
&&
\mbox{a) there are no static KIDs in $\Habc{k+4}$ , and }
\\
&&
 \mbox{b) the solution metrics are conformally compactifiable. }
 \label{-23VI14.2bh}
\eea

Now, it is well-known that static KIDs on $\Omega$  are exactly of $\tilde g$-order $z^{-1}$, so those of order $o_{\tilde g}(z^{-1})$ vanish (see Appendix~\ref{s2X15.1}).    Hence, the requirement \eq{-+23VI14.2ah}
that there are no KIDs in the space under consideration will be satisfied when $z^{-1}\not\in \Habc{0}$; equivalently
%
\bel{13VI14.1h}
 b\le \frac{n+1}{2}
 \,.
\ee

For \eq{-23VI14.2bh} we  consider  the linearized equation, as then the implicit function theorem will guarantee an identical behaviour of the full non-linear correction
to the initial data. Hence, we consider a perturbation $ \delta \rho $ of the scalar curvature on $\Omega$.
Recall that the linearized perturbed metric $  \delta g $ is obtained from the
solution $ \delta N $ of the equation
\bel{23VI14.3h}
  P_g\phi^4\psi^2  P^*_g(\delta N )=\delta \rho \in \psi^2
\Lpsikg{k}{g}
\,,
\ee
where one sets
\bel{23VI14.3h+}
  \delta g
   = \phi^4\psi^2  P^*( \delta N)\in \psi^2
\phi^2\Lpsikg{k+2}{g}
\,.
\ee
Appendix~\ref{A12VII15.1} gives
\bel{17IX15.2}
 \delta N =
o(x^{-a-n/2} z^{-b } \rho^{-c+n/2})
 \,,
\ee
with corresponding behaviour of the derivatives:
$$
\left(\frac x\rho\right)^i\nabla^{(i)}\delta N =
o_g(x^{-a-n/2} z^{-b } \rho^{-c+n/2})
 \,.
$$
This leads to
\bean  \delta g =
  o_g(x^{a-n/2+2} z^{b } \rho^{c+n/2-2})
  \,.
\eeal{13VI14.5h+}
We conclude that \eq{-+23VI14.2ah}-\eq{-23VI14.2bh} will be satisfied for all $a$ big enough and $c\ge 2-n/2$ if and only if
\bel{13VI14.6h}
   0 \le b \le \frac{n+1}2 \,,
\ee
%
where a large value of $a$ guarantees high differentiability of the tensor field $h$ when extended by zero across $\partial \Omega$.

One can now check that the conditions on the weights $a$, $b$ and $c$ guarantee the differentiability of the
map $r_\chi$ of \eq{3X15.1}. We make some comments about this in Appendix~\ref{A21XI15.1}

To end the proof, we note that Theorem~3.7 of~\cite{ChDelay} invokes Theorem~3.4 there. As such, that last theorem assumes that the inequality (3.1) of~\cite{ChDelay} holds, as needed to apply the conclusion of Proposition~3.1 there. In our context, the required conclusions of \cite[Proposition~3.1]{ChDelay} are provided
instead by 
Corollary \ref{C23V15} below (where a trivial shift of the indices $a$, $b$ and $c$ on the weight function $\psi$ has to be performed).
\qed
%

\subsection{Exchanging asymptotically hyperbolic regions}
 \label{ss18IX15.3}

The aim of this section is to show that \emph{any two} asymptotically hyperbolic metrics sharing the same conformal structure on a  $\partial M$-neighborhood of a point $p$ belonging to the conformal boundary $\partial M$ can be glued together near  $p$.
A case of particular interest arises when one of the metrics is the hyperbolic metric, leading to a configuration,  with well defined and finite total mass if this was the case for the other metric, and with an \emph{exactly hyperbolic} metric near a  $\partial M$-open subset of the conformal boundary.

\begin{theorem}
 \label{propRcassimple2}
Let $n/2<k<\infty$,  $\sigma>\frac{n }2 $. Consider two asymptotically hyperbolic metrics $g$ and $\hg$ such that $g\in M_{\zg+C^{k+4}_{1,z^{-1}}}$ and $\hg-g\in C^{k+4}_{1,z^{-\sigma}}$. For any $p\in \partial M$ and any $ b\in\big(0,\min(\frac{n+1}2, \sigma-\frac{n-1}2)\big)$ there exist $\overline M$-neighborhoods $\mcU$
and $\mcV$ of $p$ such that $\overline\mcU \subset \mcV$ and a metric
$$
 \breve g \in C^{k+2-\lfloor n/2\rfloor}\cap C^{k+2}_{1,z^{-b}}
$$
satisfying
\bel{solmodker2+}
  \breve g|_{M \setminus \mcV} = \hg
    \,,
 \quad
 \breve g|_{\mcU} =  g
    \,,
\ee
with the Ricci scalar $R(\breve g)$ of $\breve g$ between $R(g)$ and $ R(\hg)$ everywhere.
\end{theorem}

\begin{remark}
{\rm
 $R(\breve g)$ is given by an interpolation formula as in \eq{21IX15.1}, where the function $\chi $ equals one outside of $\mcV$ and vanishes in $\mcU$.
\myqed
}
\end{remark}

\begin{remark}
{\rm
In a coordinate system $(z,\theta)$ as in \eq{bgenerale}
the sets $\mcU$ and $\mcV$ will be small coordinate half-balls near $p$, and can be chosen as small as desired while preserving
a uniform ratio of their radii.
The metric $\breve g$ will be arbitrarily close to  $g$ in the sense that the norm
$
\|\breve g - g \|_{C^{k+2-\lfloor n/2\rfloor}_{1,z^{-\sigma}}} $
will tend to zero as the half-balls shrink to a point. A weighted-Sobolev estimate, as in Theorem~\ref{T21IX15.1},
for $\breve g- g$ can be read-off from the proof below.

If both metrics have a well defined hyperbolic mass, then so does $\breve g$, with the total mass of $\breve g$ approaching that of $g$ when the half-balls shrink to a point.
\myqed
}
\end{remark}

\proof
Let $r$ be a defining function for the conformal boundary such that
$r$ is the $r^2 g$-distance
to $\partial M$ near this boundary (the reason why we use the symbol $r$ for the coordinate denoted by $z$ elsewhere in this work will become apparent shortly). In particular near the boundary the metric takes the form
\bel{22IX15.11}
 g=r^{-2}(dr^2+h(r))
 \,,
\ee
where $h(r)$ is a family of Riemannian metrics on $\partial M$.

Let $ \thetaw $ denote a coordinate system on $\partial M$ centered at $p$, with $r^2g_{ij}|_pd\thetaw ^i d\thetaw ^j = \sum_i (d\thetaw ^i)^2$, and with $\partial_\thetaw   (r^2 g_{ij})|_p =0$, as can be arranged by a polynomial change of coordinates $\thetaw $.
Thus, in local coordinates, for small $r$ and $|\thetaw |$ the metric $g$ takes the form
\bea
 g
 & =
 &
\underbrace{r^{-2}\bigg( dr^2 + \sum_i (d\thetaw ^i)^2}_{=:\zg} + \big(O(r ) + O(|\thetaw |^2)\big)_{ij} d\thetaw ^i d \thetaw ^j
 \bigg)
%
\,.
\eeal{22IX15.12}
(Note that $\zg$ is the hyperbolic metric in the half-plane model.)
The metric $\hat g$ similarly takes the form \eq{22IX15.11} with $h(r)$ there replaced by $\hat h(r)$,  with
\bel{22IX15.11x}
 \hg - g= \hat h - h = O(r^{\sigma-2} )_{ij} d\thetaw ^i d\thetaw ^j)
 \,.
\ee
%
%
Let $\mcH$, $B_\delta$ and $A_{\epsilon,\delta}=B_{\delta}\setminus B_\epsilon$ be as in Section~\ref{ss18IX15.1}:
$$
 \mcH
  =\{(\theta ,z)|\ z > 0\,,\  \theta  \in\R^{n-1}\}
  \,, \
 B_\delta =   \{z>0\,,\    |\theta |^2 + z^2 <\delta
 \}
 \,,
 \
 A_{\epsilon,\delta}=B_{\delta}\setminus \overline{B_\epsilon}
 \,.
$$
Let $\phi_\lambda:A_{1,4}\longmapsto A_{\lambda,4\lambda}$ as
defined by
$$
(\thetaw, r)=\phi_\lambda(\theta,z)=(\lambda \theta,\lambda z)
 \,,
$$
and
on $B_5$ define the metrics
$$
 g_\lambda=\phi^*_\lambda  g
 \,,
 \quad
\hg_\lambda=\phi^*_\lambda \hg= g_\lambda+O_{ g_\lambda}(\lambda^\sigma z^\sigma) \,.
$$
When $\lambda$ approaches zero the metrics $g_\lambda $ and $\hg_\lambda$ uniformly approach each other in $C^{k+4}_{z,z^{-\sigma}}(A_{1,4})$, so that for all $\lambda$ small enough the result follows from Theorem~\ref{propRcassimple}.

Note that the correction tensor $h_\lambda$ obtained in this way can be estimated as
$$
h_\lambda= o_{g_{\lambda}}(\lambda^\sigma x^{a-n/2+2}z^{b}\rho^{c+n/2-2})=
o_{g_{\lambda}}(\lambda^\sigma z^{b})=o_{g}(\lambda^{\sigma-b} r^{b})\,,
$$
where in the second equalities we have used the fact that the functions $x$ and $\rho$ are bounded on $A_{1,4}$, with positive powers as $a$ and $c$ are large.
\qed

\subsection{Asymptotically hyperbolic data sets}
 \label{ss18IX15.7}

The above gluing for  scalar curvature can be extended to the full vacuum constraint operator, with or without a cosmological constant $\Lambda$.  Recall that given a pair  $(K,g)$, where $K$ is a symmetric two-tensor field and $g$ is a Riemannian metric,  the matter-momentum one form $J$ and the matter-density function $\rho$ are defined as
\be
 \label{defJrho}
 \sources(K,g)\equiv \left(
\begin{array}{c}
J\\
  \\
\rho
\end{array}
\right)
(K,g):=
\left(
\begin{array}{l}
2(-\nabla^jK_{ij}+\nabla_i\;\tr  K)\\
  \\
R(g)-|K|^2 + (\tr  K)^2-2\Lambda
\end{array}
\right),
 \ee
where the norm and covariant derivatives are defined by the metric $g$. The map $\sources$ will
be referred to as the \emph{constraint operator}.

There exist two standard settings where asymptotically hyperbolic initial data occur.
The first one is associated to the representation of hyperbolic space   as a hyperboloid  in Minkowski, in which case
$g$ is the hyperbolic metric, $K=g$ and $\Lambda=0$.
The second  one corresponds to the occurrence of hyperbolic space as a static slice in anti-de Sitter space-time, in which case
$g$ is again the hyperbolic metric, $K=0$ and $\Lambda=-\frac{n(n-1)}2$. We also note sporadic appearance of exactly hyperbolic slices of de Sitter space-time in the literature.

Choose a constant $\tau\in \R$  and  set
\bel{28IX15.1}
 \Lambda=\frac{n(n-1)(\tau^2-1)}2
 \,.
\ee
We say that a couple $(K,g)$ is  asymptotically hyperbolic, or AH,  if $g$ is an AH metric and $K$ is a symmetric two-covariant tensor field on $M$
 with $|K-\tau g|_g$
tending  to zero at the boundary at infinity. (Note that $\tau$ is \emph{not} the trace of $K$ at infinity, but $n$ times the trace. The constant $\Lambda$ is chosen so that AH initial data just defined satisfy asymptotically the vacuum constraint equations \eq{defJrho}.)

We denote by $P_{(K,g)}$ the linearization of the constraint operator at $(K,g)$.

We will make a gluing-by-interpolation. The main interest is that of vacuum data, which then remain vacuum, but there are matter models (e.g. Vlasov, or dust) where an interpolation might be of interest.
Starting  with two AH initial data sets $(K,g)$ and $(\hat K,\hat g)$, sufficiently close to each other on regions as in the previous sections, and given a cut-off function $\chi$ as before, we  define
$$(K,g)_\chi:=\chi(\hat K,\hat g)+(1-\chi)( K, g)
 \,,
$$
$$ \sources_\chi:=\chi\sources(\hat K,\hat g)+(1-\chi)
    \sources( K, g)\,,
$$
$$
 \delta\sources_\chi:=\sources_\chi
 -\sources \big((K,g)_\chi\big)
 \,.
$$
With these definitions, the gluing procedure for the constraint equations becomes a direct repetition of the one given for the scalar curvature. Letting $\Omega$ be as in Section~\ref{ss18IX15.2}, one obtains:

\begin{theorem}
 \label{Thefulltheorem}
Let $k>n/2$, $b\in[0,\frac{n+1}2]$, $\sigma>\frac{n-1}2+b$, $\phi = x/\rho$, $\psi= x^{a}y^b \rho^c$, $\tau \in \R$.  Suppose that $g\in M_{\zg+C^{k+4}_{1,z^{-1}}}$ and
$K-\tau\zg\in C^{k+3}_{1,z^{-1}}$.
For all real numbers $a$ and $c$ large enough and all
$(\hat K,\hg)$ close enough to $ (K,g)$ in $C^{k+3}_{1,z^{-\sigma}}(\Omega )\times C^{k+4}_{1,z^{-\sigma}}(\Omega )$
there exists a unique couple of two-covariant symmetric tensor field of the form
$$
  (\delta K, \delta g)
   = \Phi\psi^2 \Phi P^*_{(K,g)_\chi}(\delta Y,\delta N)\in \psi^2 \left(\phi\Lpsikg{k+2}{g}\times
 \phi^2\Lpsikg{k+2}{g} \right)
$$
such that  $(K_\chi+\delta K,g_\chi+\delta g)$ solve
\bel{fullcolle}
 \sources
    \left[(K,g)_\chi+(\delta K,\delta g)\right]
 - \sources[
(K,g)_\chi]
=\delta\sources_\chi
 \,.
\ee
Furthermore,  there exists a constant $C$ such that
\beal{estimatesolutionfullC}
 \hspace{8mm}\|(\delta K,\delta g)\|_{\psi^2\left(\phi\mathring H^{k+2}_{\phi,\psi}(g_\chi)\times \phi^2\mathring H^{k+2}_{\phi,\psi}(g_\chi)\right)}
 && \\
&& \nonumber\hspace{-3cm}\leq C \left\|
   \sources(\hat K,\hat g)-\sources (K,g)
        \right\|_{ \mathring H^{k+1}_{1,z^{-b}}(g_\chi)\times \mathring H^{k}_{1,z^{-b}}(g_\chi)}
 \,.
 \phantom{xx}
\eea
The   tensor fields $(\delta K,\delta g)$ vanish at $\partial \Omega$ and can be $C^{k+2-\lfloor\frac n2\rfloor}$-extended by zero across $\partial \Omega$.
\end{theorem}

\begin{Remark}
   \label{R28IX15.1}
   {\rm
There are obvious analogous initial-data rephrasings of  the results of Sections~\ref{ss18IX15.3} and \ref{ss18IX15.4}, we leave the detailed spelling-out of the results to the reader.}
\myqed
\end{Remark}

 \medskip

\proof
We only sketch the proof, which is essentially identical to the scalar-curvature case. We claim that we can use the implicit function theorem to solve the equation \eq{fullcolle},
where
\bel{kgegalpyn}
(\delta K,\delta g)=\psi^2\Phi^2 P^*_{(K,g)_\chi}(Y,N)
 \,,
\ee
and where $\Phi$ is defined as
\be\label{DefPhi}
 \Phi(x,y):=(\phi x,\phi^2 y)
 \,.
\ee
For this, we first need a weighted Poincar\'e type inequality near the boundaries  for $P^*_{(K,g)}$. By a direct adaptation of the arguments
given at the beginning of~\cite[Section~6]{ChDelay},
it suffices
to prove the corresponding inequality for $\phi^2 P^*_g$ and for $\phi S$. We will see shortly that the conditions imposed
on the weights imply that the operators concerned have no kernel, and thus the desired inequalities are provided by Corollary~\ref{C23V15} below with $\mathcal F=\mathring H^2_{\phi,\psi}$ and Theorem~\ref{TC22V15}  below with ${ F}=\mathring H^1_{\phi,\psi}$.
This provides the desired inequality for tensor fields supported outside of a (large) compact set .

%

Putting the  inequalities together 
 we find that on any closed space transversal to the kernel of $P^*_{(K,g)}$ it holds that
$$
\|Y\|_{\Lpsione(g)}+ \|N\|_{\Lpsitwo(g)} \le C
\| \Phi P^*_{(K,g)}(Y, N)\|_{\Lpsi(g)} \,.
$$
If we choose the weights so that there is no kernel (see Appendix~\ref{s2X15.1}),
one concludes existence of perturbed initial data $(\delta K, \delta g)$ given by \eqref{kgegalpyn}, where
$$
( Y, N)\in \Lpsikg{k+3}{g}\times
\Lpsikg{k+4}{g}
\,.
$$
Appendix~\ref{A12VII15.1} gives in particular
\bel{17IX15.2+}
  N =
o(x^{-a-n/2} z^{-b } \rho^{-c+n/2})
\,,
\quad
  Y =
o_g(x^{-a-n/2} z^{-b } \rho^{-c+n/2})
\ee
with corresponding behaviour of the derivatives:
$$
 \frac {x^i}{\rho^i}
 \nabla^{(i)} N =
 o_g(x^{-a-n/2} z^{-b } \rho^{-c+n/2})
 \,,
 \
    \frac {x^i}{\rho^i}
    \nabla^{(i)} Y =
    o_g(x^{-a-n/2} z^{-b } \rho^{-c+n/2})
 \,.
$$
In particular $\Phi P^*_{(K,g)}$ will possess the same  asymptotic behaviour.
Thus
\beal{13VI14.5h}
\hspace{10mm}\lefteqn{
  (\delta K, \delta g)
   = \Phi\psi^2 \Phi P^*(\delta Y,\delta N)\in \psi^2 \left(\phi\Lpsikg{k+2}{g}\times
\phi^2\Lpsikg{k+2}{g} \right)
   }
 &&
\\
\nonumber&&
   =
    x^{2a}z^{2b}\rho^{2c}\left(\frac{x}{\rho} o_g(x^{-a-n/2} z^{-b } \rho^{-c+n/2}) \,, \frac{x^2}{\rho^2}  o_g(x^{-a-n/2} z^{-b } \rho^{-c+n/2}) \right)
    \nn
\\
  \nonumber &   &=
	\left( o_g(x^{a-n/2+1} z^{b } \rho^{c+n/2-1}) \,,
  o_g(x^{a-n/2+2} z^{b } \rho^{c+n/2-2})\right)
  \,.
\eea
In fact,  Appendix~\ref{A12VII15.1} gives
\bel{estiKgSobolev}
  (\delta K, \delta g)\in
	C^{k+2-\lfloor\frac n 2\rfloor+\alpha}_{{x}{\rho^{-1}},x^{-a+n/2-1} z^{-b } \rho^{-c-n/2+1}} \times
  C^{k+2-\lfloor\frac n 2\rfloor+\alpha}_{{x}{\rho^{-1}},x^{-a+n/2-2} z^{-b } \rho^{-c-n/2+2}}
  \,.
\ee
One concludes using the implicit function theorem as in the proof of Theorem~\ref{T21IX15.1}.
\qed

\subsection{Localised Isenberg-Lee-Stavrov gluing}
 \label{ss18IX15.4}

The gluing constructions so-far allow one to provide an analogue of ``Maskit-type'' gluings of
Isenberg, Lee and Stavrov~\cite{ILS}. This proceeds as follows:

Consider points $p_1$, $p_2$, lying on the conformal boundary of two asymptotically hyperbolic manifolds
$(M_1,K_1,g_2)$ and $(M_2,K_2,g_2)$, or two-distinct points lying on the conformal boundary of an asymptotically hyperbolic manifold, with the same constant asymptotic value $\tau$ of the extrinsic curvature tensors in all cases, and with conformally flat boundaries at infinity.
For definiteness we sketch the construction in the former case, the latter differing from the former in a trivial way.
As described above, for all $\varepsilon>0$ sufficiently small we can replace the metrics $g_i$ by the hyperbolic metric in coordinate half-balls $\mcU_\varepsilon^{(1)}$ around $p_1$ and
$\mcU_\varepsilon^{(2)}$ around $p_2$, and the $K_i$'s by $\tau $ times the hyperbolic metric. By scaling $\mcU_\varepsilon^{(2)}$ we obtain a metric near $p_2$ which is the hyperbolic metric inside a half-ball $B_4$ and coincides with $g_2$ outside of $B_8$. Performing a hyperbolic inversion of $\mcU_\varepsilon^{(1)}$ about $p_1$, followed by a scaling, one obtains a metric on the half-space model which coincides with $g_1|_{B_1}$ inside a semi-ball $B_1$ and the hyperbolic space outside of a half-ball $B_2$. Identifying the
hyperbolic metric on the annulus  $A_{2,4}$ in the trivial way, one obtains a ``Maskit-type'' gluing of the initial data across the annulus, as illustrated in Figure~\ref{fig:1}.
\begin{figure}[h]
\hspace{-1cm}
\begin{minipage}[b]{.38\textwidth}
\begin{center}
{\includegraphics{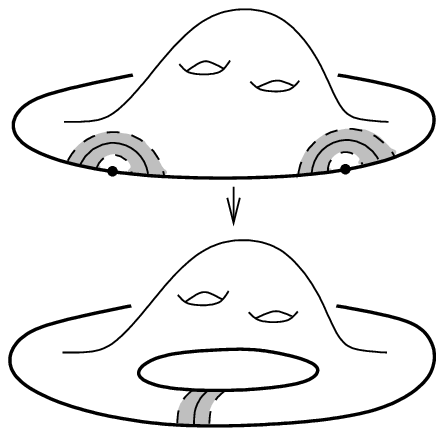}}
\end{center}
\end{minipage}
\begin{minipage}[b]{.53\textwidth}
\begin{center}
{\includegraphics{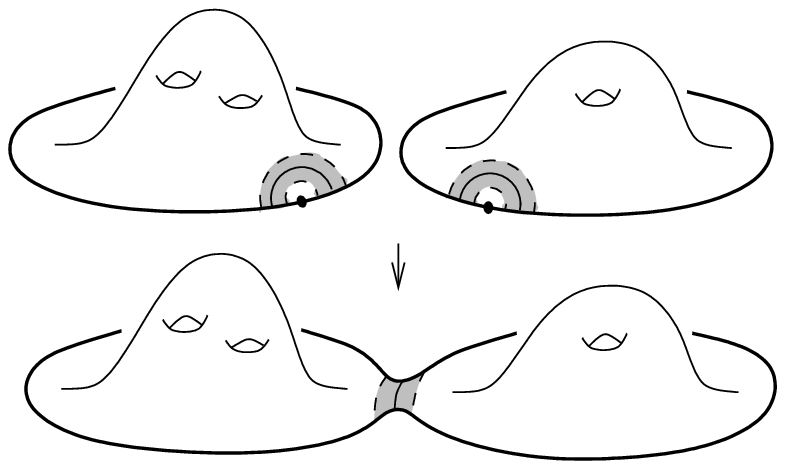}}
\end{center}
\end{minipage}
\caption{\label{fig:1}The Isenberg-Lee-Stavrov gluing of an AH initial data set to itself (left figure) or of two AH initial data sets (right figure). Figures from~\cite{ILS}, with kind permission of the authors.}
\end{figure}

Summarising, we have (the detailed regularity conditions on the metrics are as in Theorem~\ref{Thefulltheorem}):

\begin{theorem}\label{T3X15.1}
Let $(M_a, K_a, g_a)$, $a=1,2$ be two asymptotically hyperbolic and asymptotically CMC
initial data sets satisfying the Einstein \(vacuum\) constraint
equations, with the same constant asymptotic values of  $\tr_{g_1}K_1$ and $\tr_{g_2}K_2$ as $\partial M$ is approached and with locally conformally flat boundaries at infinity
$\partial\overline M_a$. Let $p_a\in \partial M_a$ be  points on the conformal boundaries.
Then for all $\varepsilon$ sufficiently small there exist asymptotically hyperbolic and asymptotically CMC vacuum initial data sets
$(M_\varepsilon, K_{\varepsilon }, G_{\varepsilon })$
such that
\begin{enumerate}
\item
$M_\varepsilon$ is diffeomorphic to the interior of a boundary connected sum of the $M_a$'s, obtained
 by excising small half-balls $B_1$ around
$p_1$ and $B_2$ around $p_2$, and identifying their boundaries.

\item
On the complement of coordinate half-balls of radius $\varepsilon$ surrounding $p_1$ and $p_2$,
and away from the
corresponding neck region in $M_\varepsilon$, the data
$(K_{\varepsilon }, g_{\varepsilon })$  coincide with the original ones
on each of the $M_a$'s.
\end{enumerate}
\end{theorem}

The analogous statement for gluing $M$ with itself,
as in the left figure of Figure~\ref{fig:1},
is left to the reader.

We emphasise that our construction applies to all dimensions, and only deforms the initial data sets near the gluing region. It does not require any non-degeneracy hypotheses, or polyhomogeneity, on the metrics. The drawback is that one has poorer control of the asymptotic behavior of the initial data in the gluing regions as compared to~\cite{ILS}; the exact differentiability of the resulting metrics can be read-off from our theorems proved above. In particular the differentiability at the conformal boundary of the final initial data set is well under the threshold needed to apply the current existence theorems for the conformal vacuum Einstein equations~\cite{F1,TimConformal,Friedrich83}.

\section{Weighted Poincar\'e inequalities}
 \label{s18IX15.6}

The aim of this section is to establish  weighted Poincar\'e inequalities for metrics of the form
\bel{19V15.4}
 g =\frac {1} {z^2}
  \big(
   {\big(1+O(z)\big)dz^2 + \underbrace{h_{i j}(z,\theta^k)d\theta^i d\theta^j}_{=:h(z)}}
    +O(z)d\theta^i dz
   \big) =: z^{-2} \tg
 \,,
\ee
where $h(z)$ is a family of Riemannian metrics satisfying
$$
c^{-1} \underbrace{\ringh_{ij} (\theta^k) d\theta^i d\theta^j}_{=:
 \ringh} \le  h(z) \le c \,\ringh
 \,,
$$
for some  smooth metric $\ringh$ and some constant $c$. Here $\{z=0\}$ is a conformal boundary at infinity, and the coordinates $(\theta^i)=(x,y^A)$ are local coordinates on the level sets of $z$. The coordinate $x$ should be thought of as a defining coordinate for another boundary $\{x=0\}$, which will be chosen to satisfy
\bel{18IX15.1}
 |dx|^2_{\tilde g}= (1+O(x))
 \quad
 \Longleftrightarrow
 \quad
 |dx|^2_g=z^2(1+O(x))
 \,.
\ee

The inequalities we are about to prove correspond to weights
$$
 \phi = \frac x \rho
 \,,
 \quad
  \psi = x^{a-1}z^b \rho^{c+1}
  \,,
$$
which are trivially shifted as compared to those of Section~\ref{SwSs}.

\begin{Remark}
   \label{R21X2015}
   {\rm
We recall that the function $z$ is assumed to be smooth, bounded, and defined globally, providing a coordinate  near the conformal boundary but not necessarily elsewhere.
Similarly the function $x$ is assumed to be smooth, bounded, and defined globally,  providing  a coordinate near its zero-level set but not necessarily elsewhere.
}
 \myqed
\end{Remark}

The following identity will often be used:
\bel{19V15.5}
 (\Gamma - \wG)^k_{ij}
 = - z^{-1} \left(2 \delta^k_{(i} \wnabla_{j)} z - \tg_{ij} \wnabla^k z
 \right)
 \,,
\ee
 \ptcr{definitions added}
where $\wnabla$ is the covariant derivative of the metric $\tg$. Further, a parenthesis over indices denotes symmetrisation, e.g.
$$
 X_{(i}Y_{j)} = \frac 12 ( X_i Y_j + Y_i X_j)
 \,.
$$

\begin{remark}{\rm
 \label{R26X15.1}
In the calculations leading to various intermediate estimates the Christoffel symbols of $\tilde g$ will be assumed to be in $L^\infty$. However, we note the following:
Many of the arguments in this paper are based on inequalities for tensor fields of the form
\bel{estiformelle}
\int|L_gu|_{g}^2\phi^2\psi^2d\mu_g\geq\int (C^2+o(1))|u|_g^2\psi^2d\mu_g,
\ee
where $L_g$ is a first order operator (equal in our case to $\nabla $ or the operator $S $ of \eq{6VII14.1} below),  and $o(1)$ tends to zero when some
parameters (e.g., a relevant variable)  tend to zero.
We observe that if the inequality  (\ref{estiformelle}) is valid for a metric $g$, and if $h$
is another metric  equivalent to $g$ and such that $|\nabla_gh|_g\phi=o(1)$, so that
$$
 |\Gamma(g)-\Gamma(h)|_g\phi=o(1)
 \,,
$$
then the inequality will remain valid for $h$.

Similar arguments can be used for inequalities involving boundary terms,
 and/or for operators of order $k$.
\myqed
}\end{remark}

\begin{remark}{\rm
 \label{R26X15.2}
Our strategy  to obtain (\ref{estiformelle}) is to use possibly several integrations by parts to obtain an identity
of the form
\bel{26X15.6}
\int u*Lu\;\phi\psi^2=\int(c^2+o(1))|u|^2\psi^2
\,,
\ee
perhaps with further boundary terms,
where $u*Lu$ is a  linear combination of controlled tensors contracted with
$u\otimes Lu$ in such a way that  the pointwise inequality $|u*Lu|_g\leq C |u|_g|Lu|_g$ holds everywhere for some constant $C $.
The inequality (\ref{estiformelle}) is then obtained by estimating the integrand of the left-hand side of  \eq{26X15.6} as
$$
 |u|_g|Lu|_g\phi\leq \frac{c^2}2|u|_g^2+\frac1{2c^2}\phi^2|Lu|_g^2
 \,,
$$
and by carrying over the $c^2 |u|_g^2\psi^2/2$ terms to the right-hand side of \eq{26X15.6}.
\myqed
}\end{remark}

We will need the following generalisation of~\cite[Proposition~C.2]{ChDelay}, with identical proof, except that a boundary term arises now:

\begin{Proposition}\label{prop:poinc}
Let $u$ be a $C^1$ compactly supported tensor field on a Riemannian manifold $(M,g)$, and
let $w,v$ be two $C^2$ functions defined on a neighborhood of the
support of $u$, then for any domain $\Omega$ with Lipchitz boundary,
\beal{14VI14.5}
  \intOmega  e^{2v}|\nabla u|^2
    &\geq&  \intOmega  e^{2v}\left[\Delta v+\Delta w +|\nabla v|^2-|\nabla
    w|^2 \right]|u|^2\\
		&&\nn
		\hspace{1cm}-\int_{\partial\Omega}e^{2v}\langle\nabla (v+w),\eta\rangle|u|^2
  \,,
\eea
where $\eta$ is the outwards unit normal to $\partial\Omega$.
\end{Proposition}

\subsection{Near the corner}
 \label{26VI16.1}

We start with the following:
\newcommand{\asigma}{a}%
\newcommand{\bmu}{b}%
\newcommand{\cgamma}{c}%

\begin{Proposition}
   \label{P21VI14.1n}
Let  $\R\ni \bmu\ne \frac{n-1}2$, $\cgamma\in \R$. For all $\asigma$ large enough
there exist constants
$\hat c>0$,
$\hat x>0$ and $\hat z>0$ such that for all differentiable tensor fields $u$ with
compact support in $\{0<x<\hat x\}\cap \{0<z<\hat z\}$ we have
\bel{14VI14.5x}
  \intOmega  x^{2 \asigma} z^{2\bmu} \rho^{2\cgamma} |\nabla u|^2
    \geq \hat c \intOmega  x^{2 \asigma-2} z^{2\bmu} \rho^{2 \cgamma+2} |u|^2
  \,.
\ee
\end{Proposition}

\begin{remark}{\rm
   \label{R23V15.2}
We can write
$$
   x^{2a} z^{2b} \rho^{2c} =  \left(\frac x z \right)^{2a} z^{2(a+b+c)}
    \bigg(1+ \frac {x^2}{ z^2}\bigg)^{ c} =: z^{2(a+b+c)} h\left(\frac x z \right)
    \,.
$$
Since $ x /z $ is equivalent to the hyperbolic distance from $\{x=0\}$, we see that our weights are equivalent to functions of the compactifying factor $z$ and of the hyperbolic distance to  $\{x=0\}$.
\myqed
}\end{remark}

\begin{remark}{\rm
  \label{R20IX15.11}
Our proof will appeal to the extensive calculations of Section~\ref{s7VII14.1}, which have to be done anyway for the purposes there. However, a computationally-friendlier  proof of Proposition~\ref{P21VI14.1n} can be carried out basing on the following observations:

First,
 \ptcr{argument added}
a standard calculation shows that it suffices to prove \eq{14VI14.5x} for functions: Indeed, suppose that \eq{14VI14.5x} with $u$ replaced by $f$ is true for all differentiable functions $f$ with support as in the statement of the theorem. Let $\chi$ be a smooth function with compact  support in $\{0<x<\hat x\}\cap \{0<z<\hat z\}$ such that $\chi\equiv 1$ on the support of the  tensor field $u$. For $\epsilon>0$ set $f_\epsilon= \chi \sqrt{\epsilon^2 + |u|^2}$. Then
$$
 |\nabla f_\epsilon|^2 \le 2\big( |\nabla \chi |^2 (\epsilon^2 + |u|^2) + |\nabla u|^2
  \big)
 \,.
$$
After passing with $\epsilon$ to zero in the inequality \eq{14VI14.5x} with $u$ replaced by the function $f_\epsilon$ one obtains \eq{14VI14.5x} for the tensor field $u$.

Next, for sufficiently small $x$ and $z$ the metric \eq{19V15.4} is equivalent
to the metric
\bel{20IX15.11}
 g =\frac {1} {z^2}
  \big(
   {dz^2 +  |d\theta|^2}
   \big)
 \,.
\ee
But when $u$ is a function, the inequality \eq{14VI14.5x} clearly implies the same inequality for any metric which is uniformly equivalent to \eq{20IX15.11}.
Hence, it suffices to prove \eq{14VI14.5x} for functions with the metric \eq{20IX15.11}. The calculations of Section~\ref{s7VII14.1} become considerably simpler
in this setting.
\myqed
}\end{remark}

{\noindent\sc Proof of Proposition~\ref{P21VI14.1n}:}
We wish to apply Proposition~\ref{prop:poinc} (actually here the original version from~\cite{ChDelay} without the boundary term suffices),
with  $u$ and $v$ equal to
%
\bel{21VI14.1}
 v = \asigma \ln x + \bmu \ln z + \cgamma \ln \rho
  \,,
\quad
   w = \nu \ln x + \beta \ln z + \lambda \ln \rho
 \,,
\ee
and where $\rho=\sqrt{x^2 + z^2}$. In fact $\lambda=0$ suffices for the current proof, but we allow $\lambda\in \R$ for future reference.
We are interested in the region of small $x$ and $z$, and therefore also small $\rho$.

Using \eq{Deltavsurv} and \eq{normedv} below one finds
\beal{20IX15.12}
 \lefteqn{\hspace{1cm}
 \Delta v+ |\nabla v|^2_{g}
 +
 \Delta w-  |\nabla w|^2_{g}
    }
     &&
\\
 \nonumber
 &=&
   \underbrace{\bmu(\bmu+1-n)+\beta(1-n-\beta)}_{=:\hat a} + O(z)+
\bigg[\underbrace{\asigma(\asigma-1) -\nu (\nu +1)}_{=:\hat b}
 +O(x)\bigg]\frac { z^2} {  x^2}
\\
 \nn
 &&
 +
  \bigg[\underbrace{ \cgamma(\cgamma+2-n-2\bmu+2\asigma)
  +    {\lambda } (-2  \nu + 2 - n - \lambda - 2 \beta) }_{=:\hat c}
\\
 &&\nn
 \phantom{+\bigg[}
 +O(x)+O(z)\bigg]\frac { z^2} { \rho^2}
   \,.
\eea
We choose $\beta=(1-n)/2$, $\nu=-1/2$, and $\lambda=0$. Then $\hat a= (\bmu+ (n-1)/2)^2>0$, $\hat b= (\asigma-1/2)^2\ge 0$ and,  for $\asigma$ large, $\hat b+\hat c x^2/\rho^2>\hat b/2>0$.
\qedskip

\subsection{A vertical stripe, $z$-weighted spaces}
 \label{ss20XI15.11}

Let $(M, g)$ an AH manifold with $g = z^{-2} {\tg}$.
Let $U$ be an open  subset of $M$ with smooth boundary $\partial U\subset M$
such that the boundary of $U$ in $\overline M$ is $\partial U\cup\partial_\infty U$,
where $\partial  U\subset   M$ and $\partial_\infty U\subset \partial M$, with $\partial U$ orthogonal to the level set of $z$ near $\{z=0\}$, in particular $\overline{\partial U}$  meets $\partial M$
$\tilde g$-orthogonally.
Let $\eta$ be the $g$-unit outwards normal to $\partial U$.
In our applications the set $U$ will be  of the form
$$
 U=\{0<x_1<x<x_2\,,\  0<z<z_1\}
$$
(smoothed-out near  $ \{z=z_1\}\times \{x=x_1   \mbox{ or }  x_2\}
$
if desired),
and note that neither $x_2$ nor $z_1$ are assumed to be small.

\begin{proposition}
For all $b\in\R$ and all tensor fields $u$ compactly supported in $M$ we have
\bea
\int_U|\nabla u|^2z^{2b}&=&\int_U
z^{2b}\left[\left(b-\frac{n-1}2\right)^2+o(1)\right]|u|^2\nn \\
&&+\int_{\partial U}
z^{2b}\left(\frac{n-1}2-b+o(1)\right)|u|^2\langle \frac{\nabla z}z,\eta\rangle.
\nn
\eea
 \ptcheck{8XI15}
\end{proposition}

\proof
We apply Proposition \ref{prop:poinc} with $v=b\ln z$ and $w=\frac{1-n}2\ln z$.
\qedskip

Given our choice of the set $U$ we have $\langle \frac{\nabla z}z,\eta\rangle=0$ near $\{z=0\}$, which leads to:

\begin{proposition}
 \label{PoincarebandeV}
 \ptcheck{8XI15}
Let $b\in\R$, $b\neq \frac{n-1}2$. There exist $z_0>0$ and a constant $C>0$ such that for all tensor fields $u$ compactly supported in $M$ it holds
$$
\int_U|\nabla u|^2z^{2b}+\int_{U\cap\{z>z_0\}}
|u|^2+\int_{\partial U\cap\{z>z_0\}}
|u|^2\geq C \int_{U}
z^{2b}|u|^2.
$$
\end{proposition}

\subsection{A horizontal stripe away from the corner}
 \label{ss8XI15.1}

In this section we stay away from the conformal boundary $\{z=0\}$, so the issue whether or not $(M,g)$ is asymptotically hyperbolic becomes irrelevant. The argument is somewhat similar to that in the last section. Note, however, that both the result and the details of the analysis here are \emph{a fortiori} different, because in the current section the metric is smooth up to the boundary $\{x=0\}$, while in Section~\ref{ss20XI15.11} the metric degenerates at $\{z=0\}$.

To make things unambiguous, throughout the current section we consider
a Riemannian manifold with boundary  $(\overline M, g)$, and denote by $M$ the interior of $\overline M$.
Let $x$ be a smooth defining function for $\partial M$, equal to the distance to
$\partial M$ near $\partial M$.
Let $U$ be an open  subset of $M$ with smooth boundary $\partial U\subset M$
such that the boundary of $U$ in $\overline M$ is $\partial U\cup\partial_0 U$,
with $\partial_0 U\subset \partial M$,  and with $\partial U$ being orthogonal to the level set of $x$ near $\{x=0\}$. In particular $\overline{\partial U}$   meets $\partial M$
$g$-orthogonally.
Let $\eta$ be the outwards-pointing unit $g$-normal to $\partial U$. Neither $\partial M$ nor $\partial U$ need to be connected.

For our applications the set $U$ will take the form
$$
 U=\{0<x<x_1\,,\ z_1<z<z_2\}\ \mbox{ smoothed-out near}\ \{z=z_1,z_2\}\times \{x=x_1\}
  ,
$$
with $x_1,z_1>0$.
As before, neither the $z_i$'s nor $x_1$ need to be small.

\begin{remark}{\rm
 \label{10XI15.1}
 The reader is warned that the results of the current (sub)section
are proved for a general Riemannian metric $g$, but will be used
in our applications with the metric $\tilde g=z^2g$ in place of $g$.
Since $z_1>0$, we are \emph{away from the corner $x=z=0$}, so the asymptotic estimates both for $g$ and $\tg$ are of the same type because
the norms are equivalent, and we also have $\Gamma_g-\Gamma_{\tilde g}=O_{ g}(1)=O_{\tilde g}(1) $
 (with constants which degenerate as $z_1$ tends to zero; this issue is addressed in our final argument by the analysis of Section~\ref{26VI16.1}).
 \myqed
}\end{remark}
\newcommand{\amoinsun}{{\color{red} (a-1)}}

\begin{proposition}
 \label{P21XI15.1}
    \ptcheck{8XI15}
For all $a\in\R$ and all tensor fields $u$ compactly supported in $M$ we have
\bea
\int_U|\nabla u|^2x^{2a}&=&\int_U
x^{2(a-1)}\left[\left(a-\frac{1}2\right)^2+o(1)\right]|u|^2\nn \\
&&-\int_{\partial U}
x^{2(a-1)}\left(a-\frac{1}2+o(1)\right)|u|^2\langle {\nabla x},\eta\rangle.
\nn
\eea
\end{proposition}
\proof
We apply Proposition \ref{prop:poinc} with $v=a\ln x$ and $w=-\frac{1}2\ln x$.
\qedskip

Since $\langle \frac{\nabla z}z,\eta\rangle=0$ for small $z$, we obtain:

\begin{proposition}
 \label{PoincarebandeH}
    \ptcheck{8XI15}
For all $a\in\R$, $a\neq \frac{1}2$, there exist $x_0>0$ and a constant $C>0$ such that for all tensor fields $u$ compactly supported in $M$,
$$
\int_U|\nabla u|^2x^{2a}+\int_{U\cap\{x>x_0\}}
|u|^2+\int_{\partial U\cap\{x>x_0\}}
|u|^2\geq C \int_{U}
x^{2(a-1)}|u|^2.
$$
\end{proposition}

\subsection{The global inequality}
 \label{ss21XI15.2}
	
The series of estimates above can be put together to obtain the desired \emph{weighted Poincar\'e inequality}:

	\begin{theorem}
  \label{T21XI15.1}
   Let $\phi= x/\rho$, $\psi=x^{a-1} z^b \rho^{c+1}$, and suppose that $F \subset \mathring H^1_{\phi,\psi}$ is a
   closed subspace of $ \mathring H^1_{\phi,\psi}$ transverse to the kernel of $\nabla$. Then
 for all  $b\neq (n-1)/2$, all $c\in\R$ and for all constants $a$ sufficiently large there exists a  constant $C_1(a,b,c, F)>0$ such that the inequality
\bel{inegapoincareglobal}
	\|\phi \nabla u\|_{L^2_\psi}\geq C_1(a,b,c,F) \|u\|_{  H^1_{\phi,\psi}}
\ee
	holds for all $u\in F$.
\end{theorem}

\begin{proof}
The proof is a repetition of that of Theorem~\ref{TC22V15} below with the following changes:
 $S$ there is replaced by $\nabla$;  $Y$ there is replaced by $u$;
Proposition~\ref{prop:estiS2ah2015} there is replaced by Proposition~ \ref{P21VI14.1n}; Proposition 5.1 of~\cite{ChDelay} there is replaced by Proposition~C.3 of~\cite{ChDelay}; and Proposition 6.1 of~\cite{ChDelay} there is replaced by Proposition~C.7 of~\cite{ChDelay}; Propositions \ref{propKornbande} and \ref{propKornbandex} there are replaced by Propositions \ref{PoincarebandeV} and \ref{PoincarebandeH}.
\end{proof}

For  future reference, we note that one of the steps of the argument just outlined proves the inequality
\bel{inegapoincareglobalmodK-}
	\|\phi \nabla u\|_{L^2_\psi}
    +\|u\|_{L^2 (K)}
    +\|u\|_{L^2 (\Sigma)}
    \geq C_1(a,b,c,K,\Sigma) \|u\|_{  H^1_{\phi,\psi}}
 \,,
\ee
for some sufficiently large  compact set $K$ and for some relatively compact smooth hypersurface $\Sigma$,
and for all tensor fields $u$ in $\mathring H^1_{\phi,\psi}$,
with a constant $C_1$ depending upon the arguments listed; compare~\eq{22V15.25} below. We note that one can get rid of the
    $\|u\|_{L^2 (\Sigma)}$-term in \eq{inegapoincareglobalmodK} using a trace theorem: letting $K_1$ be a compact neighborhood of $\Sigma$, there exists a constant $C$ such that
$$
 \|u\|_{L^2(\Sigma)}
  \le C
 \|u\|_{H^1(K_1)}
  \,.
$$
Using again the symbol $K$ for $K\cup K_1$, trivial rearrangements in \eq{inegapoincareglobalmodK-} lead to
  \bel{inegapoincareglobalmodK}
	\|\phi \nabla u\|_{L^2_\psi}
    +\|u\|_{L^2 (K)}
    \geq C_1(a,b,c,K) \|u\|_{  H^1_{\phi,\psi}}
 \,.
\ee

\section{Weighted Korn inequalities}
 \label{s7VII14.1}

Given a vector field $Y$, let $S(Y)$ denote one-half of the Killing form of $Y$:
\bel{6VII14.1}
 S(Y)_{ij}:= \frac 12 (\nabla_i Y_j + \nabla_j Y_i)
 \,.
\ee

We will need the following results, where $\Omega$ is allowed to be any set with piecewise differentiable boundary, and where the vector field $Y$ is allowed to be non-vanishing on $\partial \Omega$. Identities \eq{7VIII14.1} and \eq{20V15.1} below are established by keeping track of the boundary terms  in the corresponding calculations in~\cite{ChDelay}.

\begin{prop}{\cite[Proposition~D.2]{ChDelay}}
 \label{prop:estiPS}
For all functions $u$, all vector fields $V$ and all vector fields
$Y$ with compact support we have the equality
\begin{eqnarray}
 \label{7VIII14.1}
\lefteqn{ \intOmega  e^{2u}[S(Y)+\frac{1}{2}\tr(S(Y))g](Y,V)}&&
\\ &&=
-\frac{1}{2}\intOmega  e^{2u}\big\{\nabla
V(Y,Y)+\frac{1}{2}\divr(V)|Y|^2+\langle  du,V\rangle_g   |Y|^2
\nonumber\\
&&
+2\langle  du,Y\rangle_g   \langle V,Y\rangle_g   \big\}
+\frac14\int_{\partial\Omega}e^{2u}(|Y|^2\langle  \eta,V\rangle_g+2\langle  Y,V\rangle_g\langle  Y,\eta\rangle_g)\,.
 \nn
\end{eqnarray}
\end{prop}

\begin{prop}{\cite[Proposition~D.3]{ChDelay}}
 \label{prop:estiPS2.2015}
For all differentiable vector fields $Y$ with compact support and functions $u$
and $v$ defined in a neighborhood of the support of $Y$ it holds:
\begin{eqnarray}\label{20V15.1}
\lefteqn{\hspace{1cm}
 -2\intOmega  ve^{2u}S(Y)(\nabla v,\nabla v)\langle dv,Y\rangle_g =
 \hspace{5cm}}
&&
\\&&\intOmega  e^{2u}\langle dv,Y\rangle_g
 \bigg[
 \langle dv,Y\rangle_g   (|dv|^2+v\Delta
    v+2v\langle dv,du\rangle_g   )
 +2v\nabla\nabla
 v(Y,\nabla v)\bigg]
 \nn
\\
 &&\nn-\int_{\partial\Omega}ve^{2u}\langle dv,Y\rangle^2_g\langle dv,\eta\rangle_g
 \,.
\end{eqnarray}
\end{prop}

Throughout this section we consider  metrics of the form \eq{19V15.4}, as in Section~\ref{s18IX15.6}.

We start with the Korn inequality for vector fields supported near the corner $\{x=z=0\}$.

Similarly to Section~\ref{s18IX15.6}, the weighted Korn inequalities that we are about to prove correspond to weights
$$
 \phi = \frac x \rho
 \,,
 \quad
  \psi = x^{a-1}z^b \rho^{c+1}
  \,.
$$

\subsection{Near the corner}
\begin{prop}
 \label{prop:PSah2015}
 For all $a,b,c,A,B,C\in\R$ and
 all vector fields $Y$ with compact support let
\bel{19V15.2}
 V:=A\frac{\nabla x}{x}+B\frac{\nabla z}{z}+C\frac{\nabla \rho}{\rho}
 \,,
 \quad
Y_x:=\langle
\frac{\nabla x}{z} \,,Y\rangle\,,\quad
    Y_z:=\langle
    \frac{\nabla z}{z} \,,Y\rangle\,.
\ee
We have
\begin{eqnarray*}
\intOmega  x^{2a}z^{2b}\rho^{2c}
[S(Y)+\frac{1}{2}\tr(S(Y))g](Y,V)=\hspace{4cm}\mbox{ }\\
\;\;\; -\frac{1}{2}\intOmega
x^{2a-2}z^{2b}\rho^{2c+2}\left.\Big(\right.{\mathcal A}|Y|^2+{\mathcal A}_{xx}Y_x^2+{\mathcal A}_{zz}Y_z^2+{\mathcal A}_{xz}Y_xY_z
\left.\right.\Big)
 \,,
\end{eqnarray*}
with
\begin{eqnarray*}
{\mathcal A}&:=&
-\frac12(1+n-2b)B\frac{x^2}{\rho^2}
 +\frac12(2a-1)A\frac{z^2}{\rho^2}
\\
 &&+\left(-\frac n2 C+[Cc+(Ac+Ca)+(Bc+Cb)]\right)\frac{z^2x^2}{\rho^4}
\\
 &&
  + O(\frac{zx^2}{\rho^2}) + AO(x)+(Ab+Ba)O(x)
  \,,
\\
{\mathcal A}_{xx}
    &:=&A(2a-1)\frac{z^2}{\rho^2}+2C(c-1)\frac{z^2x^4}{\rho^6}
 +2(Ac+C(a{-}1))\frac{z^2x^2}{\rho^4}
\\
&&
 + O(\frac{zx^2}{\rho^2})+AO(\frac{xz^2}{\rho^2})
  \,,
\\
{\mathcal A}_{zz}
 & := &
 O(\frac{x^2}{\rho^2})
  \,,
\\
{\mathcal A}_{xz}
 & := &O(\frac{xz}{\rho^2})
 \,.
\end{eqnarray*}
\end{prop}

\begin{proof}
We use Proposition~\ref{prop:estiPS}  with
\bel{19V15.3}
u=a\ln(x)+b\ln(z)+c\ln(\rho)
 \,.
\ee
It follows from \eq{19V15.4} that
\bel{19V15.1}
\mbox{
$|dz|^2_g=z^2(1+O(z))$, $|dx|^2_g=z^2(1+O(x))$, and  $\langle \nabla x,\nabla z\rangle_g =O(z^3)$.}
\ee
Then
%
\bean
 \nabla\nabla x & = & \frac1z(\nabla z\nabla x+\nabla x\nabla z
)+O_g(z^2) \,,
\\
\label{hesslnx}
\nabla(x^{-1}\nabla x)
 & = &
  -\frac{\nabla x\nabla x}{x^2}+\frac1{xz}(\nabla z\nabla x+\nabla x\nabla z
)+O_g(\frac{z^2}{x})
     \,,
\\
\nabla\nabla (z^{-1})
 &= &
  z^{-1}g+O_g(1)
     \,,
   \nn
\\
 \nn
    \nabla\nabla z
     & = & -zg+2\frac{\nabla z\nabla z}{z}+O_g(z^2)
     \,,
\\
      \label{hesslnz}
\nabla(z^{-1}\nabla z)
 & = & -g+\frac{\nabla z\nabla z}{z^2}+O_g(z)
     \,,
\eea
and ${\rho}{\nabla\rho}=x\nabla x+z\nabla z$.
%
%
We deduce
%
\bea\label{hesslnrho}
\;\;\nabla(\rho^{-1}\nabla\rho)
 &=&
  \left(-2\frac{x^2}{\rho^4}+\frac 1{\rho^2}\right)\nabla x\nabla x
   +\left(-2\frac{z^2}{\rho^4}+\frac {3}{\rho^2}\right)
\nabla z\nabla z\\
\nonumber
&&\nn+\left(-2\frac{xz}{\rho^4}+\frac x{z\rho^2}\right)(\nabla x\nabla z+\nabla z\nabla x)\\
&&\nn+\left(
-\frac {z^2}{\rho^2}\right)g+O_g(\frac{xz^2}{\rho^2})
+O_g(\frac{z^3}{\rho^2})
 \,,
 \eea
\beal{nablaV}
\nabla V&=&\left[-A\frac{1}{x^2}+C\left(-2\frac{x^2}{\rho^4}+\frac 1{\rho^2}\right)\right]\nabla x\nabla x\\
\nonumber
&&+\left[B\frac{1}{z^2}+C\left(-2\frac{z^2}{\rho^4}+\frac {3}{\rho^2}\right)\right]\nabla z\nabla z\\
\nonumber
&&
+\left[A\frac{1}{xz}+C\left(-2\frac{xz}{\rho^4}+\frac x{z\rho^2}\right)\right](\nabla x\nabla z+\nabla z\nabla x)
\\
 \nonumber
 &&
 +\left[
-B+C\left(
-\frac {z^2}{\rho^2}\right)\right]g\\
&&\nn+O_g(z)+AO_g(\frac{z^2}x)+O_g(\frac{xz^2}{\rho^2})
+O_g(\frac{z^3}{\rho^2})
 \,,
%
%
\eea
\beal{divV-1}
\mbox{div} V
 &=&
 -A\frac{z^2}{x^2}+(1-n)B+(2-n)C\frac{z^2}{\rho^2}+O(\frac{x^3z^2}{\rho^4})+O(\frac{z^3}{\rho^2})
\\
&&\nn+O(z)+AO(\frac{z^2}x)+O(\frac{xz^2}{\rho^2})
+O(\frac{xz^4}{\rho^4})
+{ O_g(\frac{z^5}{\rho^4})}
 \,.
\eea
Inspection of the error terms allows us to rewrite the last equation as
%
\beal{divV}
\mbox{div} V
 &=&
  -A\frac{z^2}{x^2}+(1-n)B+(2-n)C\frac{z^2}{\rho^2}\\
&&\nn
+O(z)+AO(\frac{z^2}x) \,.
\eea
Using
%
\beal{nablauV}
 V \nabla u&=&\left[aA\frac{1}{x^2}+cC\frac{x^2}{\rho^4}+(Ac+Ca)\frac 1{\rho^2}\right]\nabla x\nabla x\\
\nonumber
&&+\left[Bb\frac{1}{z^2}+Cc\frac{z^2}{\rho^4}+(Bc+Cb)\frac {1}{\rho^2}\right]\nabla z\nabla z\\
\nonumber
&&
+\left[Ab\frac{1}{xz}+Ac\frac{z}{x\rho^2}+Cb\frac x{z\rho^2}+Cc\frac{xz}{\rho^4}\right]\nabla x\nabla z
\\
&&\nn
+\left[Ba\frac{1}{xz}+Bc\frac{x}{z\rho^2}+Ca\frac z{x\rho^2}+Cc\frac{xz}{\rho^4}\right]\nabla z\nabla x
     \,,
\eea
%
we obtain
\beal{tracenablauV}
 \langle V, \nabla u\rangle_g 
    &=&
     aA\frac{z^2}{x^2}+Bb+[Cc+ Ac+Ca + Bc+Cb ]\frac{z^2}{\rho^2}
\\
    &&\nn
     + O(z) + AO(\frac{z^2}x)+(Ab+Ba)O(\frac{z^2}x)
     \,.
\eea
Let
$$
 -2I := \frac 12  {\mbox{div} V}  + \langle du,V\rangle_g  - \left(
 B+C \frac {z^2}{\rho^2}\right)
$$
denote the sum of terms which contribute to the multiplicative factor of $|Y|^2e^{2u}$ in the integrand of the right-hand side of \eq{7VIII14.1}  (with the last term above arising from $\nabla V(Y,Y)$, compare the last but one line of \eq{nablaV}). We find
%
\bean
-2I
 & = &
 \frac{\rho^2}{x^2}
  {\mathcal A}
  \,.
\eea

Let $-2 II$ denote the multiplicative factor in front of $e^{2u} z^{-2}\langle Y,\nabla x\rangle_g ^2$  in the integrand of the right-hand side of \eq{7VIII14.1}; such terms arise from $\nabla V (Y,Y)$ and $2\langle du, Y \rangle_g  \langle V,Y\rangle_g $ there:
%
\beaa
-2II
  & = &
  z^2\left[-A\frac{1}{x^2}+C\left(-2\frac{x^2}{\rho^4}+\frac 1{\rho^2}\right)\right]
\\
 &&
     +2z^2\left[aA\frac{1}{x^2}+cC\frac{x^2}{\rho^4}+(Ac+Ca)\frac 1{\rho^2}
 \right]
 +O(z)+AO(\frac{z^2}{x})
\\
 &=: &
   \frac{\rho^2}{x^2} {\mathcal A}_{xx}
 \,.
\eeaa

Let
$-2 III
$
denote the multiplicative factor in front of $e^{2u} z^{-2}\langle Y,\nabla z\rangle_g ^2$  in the integrand of the right-hand side of \eq{7VIII14.1}; such terms arise from $\nabla V (Y,Y)$ and $2\langle du, Y \rangle_g  \langle V,Y\rangle_g $ there:
\bean
-2III
&=&\frac{\rho^2}{x^2}\frac{x^2}{\rho^2}
\left[B+C\left(-2\frac{z^4}{\rho^4}+\frac {3z^2}{\rho^2}\right)+2(b+\frac{z^2}{\rho^2}c)(B+\frac{z^2}{\rho^2}C)
 +\lots \right]
 \\
 & = &\nn \frac{\rho^2}{x^2} \times O(\frac{x^2}{\rho^2})
  =:  \frac{\rho^2}{x^2}
  {\mathcal A}_{zz}
 \,;
\eea
here, and in what follows, the lower order terms denoted by $\lots$ have a structure similar to those already encountered in the other terms above and are dominated by the remaining terms present in the equation.

Let $-2 IV$ denote the multiplicative factor in front of
$e^{2u} z^{-2}\langle Y,\nabla x\rangle_g \langle Y,\nabla z\rangle_g $
in the integrand of the right-hand side of \eq{7VIII14.1};
such terms arise again from $\nabla V (Y,Y)$ and $2\langle du, Y \rangle_g  \langle V,Y\rangle_g $ there:
\bean
-2IV
 &= &
 \frac{\rho^2}{x^2}\frac{x^2}{\rho^2}\left[2z^2\left(A\frac{1}{xz}+C\left(-2\frac{xz}{\rho^4}+\frac x{z\rho^2}\right)\right)\right.
\\
 &&
 \left.+2(a\frac z x+c\frac{xz}{\rho^2})(B+C\frac{z^2}{\rho^2})
 +2(A\frac z x+C\frac{xz}{\rho^2})(b+c\frac{z^2}{\rho^2})
 + \lots \right]
 \nn
 \\
 & =&\nn \frac{\rho^2}{x^2}\times  O\left(\frac{xz}{\rho^2} \right)
  =:
   \frac{\rho^2}{x^2}{\mathcal A}_{xz}
 \,.
\eea
Adding ends the proof of Proposition~\ref{prop:PSah2015}.
%
%
\end{proof}

We continue with the following:

\begin{cor}
 \label{c20V15.1}
For all differentiable compactly supported vector fields $Y$ we have
\bean
\nn\lefteqn{
 -2\intOmega x^{2a}z^{2b}\rho^{2c}S(Y)\left(\frac {\nabla z}z,\frac {\nabla z}z\right)\langle Y,\frac {\nabla z}z\rangle_g
  }
  &&
 \\ \nn
  && =
  \intOmega x^{2a-2}z^{2b}\rho^{2c+2}
   \left((1-n+2b +O(z))\frac{x^2}{\rho^2}+2c\frac{x^2z^2}{\rho^4} \right)
\langle \frac{\nabla z}z,Y\rangle_g ^2.
 \phantom{xxxx}
\eea
\end{cor}

\begin{proof}
For further reference,
we consider again (see (\ref{19V15.2}))  the vector field $V=\frac{\nabla v}v$ , where
$$
v=x^Az^B\rho^C \,,
$$
and
$$
u=a\ln x+b\ln z+c\ln\rho
 \,.
$$
However, for the strict purpose of the current proof, only the case $A=C=0$ is needed.

Using \eq{nablaV} together with \eq{nablauV} with $  u$ there replaced by $\ln v$, we find
%
\beal{hessvsurv}
 \lefteqn{\;\;
 \frac{\nabla\nabla v}{v}=\nabla V+VV
 }
 &&
\\
 \nonumber
 &=&
 \left[(A^2-A)\frac{1}{x^2}+(C^2-2C)\frac{x^2}{\rho^4}+(C+2AC)
      \frac 1{\rho^2}\right]\nabla x\nabla x
\\
\nonumber
&&+\left[(B^2+B)\frac{1}{z^2}+(C^2-2C)\frac{z^2}{\rho^4}+(3C+2BC)\frac {1}{\rho^2}\right]\nabla z\nabla z\\
\nonumber
&&
+\left[(AB+A)\frac{1}{xz}+(C^2-2C)\frac{xz}{\rho^4}+(BC+C)\frac x{z\rho^2}+AC\frac{z}{x\rho^2}\right]
\\
 \nn
 &&  \phantom{+\bigg[} \times
 (\nabla x\nabla z+\nabla z\nabla x)
\\
 \nonumber
 &&+\left[
-B+C\left(
-\frac {z^2}{\rho^2}\right)\right]g\\
&&\nn+O_g(z)+AO_g(\frac{z^2}x)+
 \underbrace{ O_g(\frac{xz^2}{\rho^2})  +O_g(\frac{z^3}{\rho^2})}_{O_g(z)}
 \,.
\eea
Taking the trace one obtains
%
\beal{Deltavsurv}
\;\;\;\;v^{-1}\Delta v
 &=&(A^2-A)\frac{z^2}{x^2}+[B^2+(1-n)B]
 \\
&&\nn
+[C^2+2AC+2BC+(2-n)C]\frac{z^2}{\rho^2}+O(z)+AO(\frac{z^2}x)
 \,.
\eea
As such, the term
$$
v^{-1}\langle \nabla v,\nabla u\rangle_g =\langle V,\nabla u\rangle_g
$$
can be read off from \eq{tracenablauV}.
Finally, it holds that
%
\bel{normedv}
 v^{-2}|dv|^2_g
  =
   |V|^2_g=A^2\frac{z^2}{x^2}+B^2+[C^2+2(AC+BC)]\frac{z^2}{\rho^2}
    +O(z)+AO(\frac{z^2}x)
 \,.
\ee

Using all of the above with $A=C=0$ in
the integrand of the right-hand side of \eq{20V15.1} one finds
%
$$
x^{2a}z^{2b+4B}\rho^{2c}B^2\langle \frac{\nabla z}z,Y\rangle_g ^2B\left(4B+1-n+2b+2c\frac{z^2}{\rho^2}+O(z)
    \right)
 \,.
$$
Replacing $b$ by $b+2B$ gives the result.
\end{proof}

\begin{prop}
 \label{prop:estiS2ah2015}
For all  $b\neq (n+1)/2, (n-1)/2$, $c>-|n-1-2b|$ and for all $a$ sufficiently large  there exist constants
$C(a,b,c)>0$,
$x(a,b,c)>0$ and $z(a,b,c)>0$ such that for all differentiable vector fields $Y$ with
compact support in $\{0<x<x(a,b,c)\}\cap \{0<z<z(a,b,c)\}$ we have
\bel{22V15.21}
\intOmega  x^{2a}z^{2b}\rho^{2c}|S(Y)|^2\; \wasdmug \geq C(a,b,c)\intOmega
x^{2a-2}z^{2b}\rho^{2c+2}|Y|^2\;\wasdmug \,.
\ee
\end{prop}
\begin{proof}
From Corollary \ref{c20V15.1}\,,  for all $\lambda>0$, we have
\begin{eqnarray*}
\nn\lefteqn{ \frac{\lambda}{2}\intOmega
x^{2s}|S(Y)|^2+\frac{1}{2\lambda}\intOmega x^{2a}z^{2b}\rho^{2c}|\nabla z/z|^2Y_z ^2}&&
\\
  &&\geq\left|\intOmega  x^{2a-2}z^{2b}\rho^{2c+2}\left[(n-1-2b+O(z))\frac{x^2}{\rho^2}+c
 \frac{x^2z^2}{\rho^4}\right]Y_z ^2 \right|.
\\
 &&
 \geq{\mathcal C}_{zz}^2\intOmega  x^{2a-2}z^{2b}\rho^{2c+2}\frac{x^2}{\rho^2}
Y_z ^2,
\end{eqnarray*}
for some constant $\mathcal C_{zz}>0$ because $c>-|n-1-2b|$.
We conclude by using Proposition~\ref{prop:PSah2015} with $B=2b-n-1$, $A>0$, $C=0$, as well as the elementary identities
$$
|[S(Y)+\frac{1}{2}\tr(S(Y))g](Y,V)|\leq
\frac{\gamma}{2}|S(Y)+\frac{1}{2}\tr(S(Y))g|^2+\frac{1}{2\gamma}|Y|^2|V|^2,
$$
for all $\gamma>0$ (chosen large here), together with $|V|^2=O(\frac{\rho^2}{x^2})$,
and
\bean
 &
|S(Y)|^2\geq \frac{1}{n}|\tr S(Y)|^2,
    \quad
    \mathcal A>{\mathcal C}^2>0\,,\;\;\;{\mathcal A}_{xx}>{\mathcal C}_{xx}^2\frac{z^2}{x^2}>0
    \,,
     &
\eea
for some constants ${\mathcal C} \,, \;{\mathcal C}_{xx}$, and
$$2|xzY_zY_x|\leq \frac{\beta}{2}z^2Y_z^2+\frac{1}{2\beta}x^2Y_x^2
 \,,
$$
for all $\beta>0$ (chosen large here).
\end{proof}

\subsection{A vertical stripe, $z$-weighted spaces}
 \label{KornbandeV}

In the current section we work in a setup identical to that of Section~\ref{ss20XI15.11}.

We start with some integration-by-parts identities. The error terms $o(1)$ in all  identities that follow in this section represent functions which tend to zero as $z$ does.

\begin{lemma}{}\label{Lcord10jeremie}
For all $Y$ with compact support in  $M$  and for all
$  b \in \R$ we have
\begin{eqnarray}\label{26X15.1}
\lefteqn{
 2\int_{U} z^{2b }( {S}(Y) + \frac12 tr_{g}( {S}(Y))g)(Y,\frac{\nabla z}{z}) \, \wasdmug
 }
 &&
 \\
\nn
& =& \int_{U}z^{2b } \Big\{ (\frac{n+1}{2}- b  + o(1)) |Y|_{g}^{2} -(2 b  + 1) \langle \frac{\nabla z}{z} , Y \rangle^{2}_{g} \Big\} \, \wasdmug   \\
&& \nn+ \frac{1}{2} \int_{{\partial U}} z^{2b } |Y|_{g}^{2} \; \langle \frac{\nabla z}{z} , \eta \rangle_{g} \, \wasdmug   + \int_{{\partial U}} z^{2b } \langle Y , \frac{\nabla z}{z} \rangle_{g} \; \langle Y , \eta \rangle_{g} \, \wasdmug
 \,.
\end{eqnarray}
\end{lemma}

\proof
We apply Proposition \ref{prop:estiPS} with $e^{u}=z^b$  and
$V=\frac{\nabla z}z$.
\qedskip

In the  applications that we have in mind the scalar product $ \langle \frac{\nabla z}{z} , \eta \rangle_{g}$ will
be zero on the vertical boundary for small $z$, and of order one on the horizontal one.
For the purpose of our final estimate we need to get rid of the  vertical boundary term in the last integral above.

It is simple to show that a weighted norm of $S(Y)$ can be used to control a weighted norm of $Y(z)$,  up to a boundary term:

\begin{lemma}{}\label{Lcord11jeremie}
For any $Y$ with compact support in  $M$ and for all
$ b \in \R$ it holds that
\begin{eqnarray} \label{26X15.2}
 \lefteqn{\hspace{1cm}
 2 \int_{U} z^{2b } {S}(Y)(\frac{\nabla z}{z},\frac{\nabla z}{z}) \; \langle \frac{\nabla z}{z} , Y \rangle_{g} \, \wasdmug
 }
  &&
\\
&= & \nn\int_{U} z^{2b }  \left(n-1- 2b  + o(1) \right) \; \langle \frac{\nabla z}{z} , Y \rangle^{2}_{g}  \wasdmug   + \int_{{\partial U}} z^{2b } \; \langle \frac{\nabla z}{z} , Y \rangle^{2}_{g} \langle \frac{\nabla z}{z} , \eta \rangle_{g}  \wasdmug
\,.
\phantom{xxxx}
\end{eqnarray}
In particular, if $\nabla z$ is orthogonal to $\partial U$ for  $0<z<z_2$ for some $z_2$, then  for each $b\ne (n-1)/2$ there exists   a constant $C$ such that
\be
   \int_{U} z^{2b } |{S}(Y) |^2 \wasdmug + \int_{\partial U\cap\{z\ge z_2\}} z^{2b } \; \langle \frac{\nabla z}{z} , Y \rangle^{2}_{g}  \wasdmug
\ge C \int_{U} z^{2b }   \langle \frac{\nabla z}{z} , Y \rangle^{2}_{g}  \wasdmug
\,.
 \label{26X15.5}
\ee

\end{lemma}

\proof
\Eq{26X15.2} is obtained by applying Proposition~\ref{prop:estiPS2.2015} with $u= (b+2)\ln z$, $v=z^{-1}$ and $\Omega=U$. Alternatively, integrate the divergence
of $z^{2b } \; \langle \frac{\nabla z}{z} , Y \rangle^{2}_{g} \frac{\nabla z}{z}$ over $U$.
\Eq{26X15.5} follows  from \eq{26X15.2} by elementary manipulations.
\qedskip

To continue, let $f$ be a defining function for $\partial U$, $f > 0$ on $U$, such that
\bean
 &
|df|_{\tilde g}^2=1, \mbox{ near $\partial U$ and }
 &
\\
 &
\langle \tilde\nabla f,\tilde\nabla z\rangle_{\tilde g}=0 \mbox{ on $\partial U$ near } \{z=0\},\mbox{ as well as on } \partial_\infty U\subset\{z=0\}.
 &
\eeal{21X15.1}
We  have (see \eqref{19V15.5})
 \ptcheck{checked 19 X}
\beal{21X15.2}
 &
|df|_{g}^2=z^2(1+O(f))\,, \quad z^{}\nabla\nabla f=\nabla z\nabla f+\nabla f\nabla z+o_g(z^2)\,,
 &
\\
 \nn &
\langle \nabla f,\nabla z\rangle_{ g}=o(z^2)
 \,,
 \quad
  \Delta_g f = o(z)
 \,,
 \quad
  \eta = - \frac 1 z \nabla f
  \, .
 &
\eea
We also recall that
 \ptcheck{checked 19 X}
$$
 |dz|^2=z^2+o(z^2)\,, \quad
\nabla\nabla z=-zg+2z\frac{\nabla z}z\frac{\nabla z}z+o_g(z)
 \,,
 \quad
 \Delta_g z = (2-n)z + o(z)
 \,.
$$
%


\begin{lemma}{}\label{LSdeuxbordprop}
 \ptcheck{checked and corrected 23 X}
\begin{eqnarray} \label{24X15.1}
 \lefteqn{\;\;\;
 - 2\int_{U} z^{2b } S(Y)(\frac{\nabla z}{z},\frac{\nabla f}{z}) \; \langle \frac{\nabla f}{z} , Y \rangle_{g} \, \wasdmug
- \int_{U} z^{2b } {S}(Y)(\frac{\nabla f}{z},\frac{\nabla f}{z}) \; \langle \frac{\nabla z}{z} , Y \rangle_{g} \, \wasdmug
 }
  &&
\\
 \nn
& =&
 \int_{U} z^{2b }
 \left[
   \left(b-\frac{n-3}2 \right) \; \langle \frac{\nabla f}{z} , Y \rangle^2_{g}
   +
    \langle \frac{\nabla z}{z} , Y \rangle^2_{g}
+ o(1)|Y|^2
 \right]
 \, \wasdmug
\\
 &&\nn- \frac12\int_{{\partial U}} z^{2b } \;
 \langle \frac{\nabla f}{z} , Y \rangle^2_{g}\langle \frac{\nabla z}{z} , \eta \rangle_{g} \, \wasdmug
-\int_{{\partial U}} z^{2b } \; \langle \frac{\nabla z}{z} , Y \rangle_{g}
 \langle\eta , Y \rangle_{g} \, \wasdmug
  \,.
   \phantom{xxxx}
\end{eqnarray}
\end{lemma}

\proof
First, we integrate the divergence of $\langle \frac{\nabla z}{z} , Y \rangle_{g}
 \langle \frac{\nabla f}{z} , Y \rangle_{g}z^{2b-1}\nabla f$, which gives
\begin{eqnarray*}
 \nn\lefteqn{
  - \int_{U} z^{2b } {\nabla}Y(\frac{\nabla f}{z},\frac{\nabla z}{z}) \; \langle \frac{\nabla f}{z} , Y \rangle_{g} \, \wasdmug
- \int_{U} z^{2b } {S}(Y)(\frac{\nabla f}{z},\frac{\nabla f}{z}) \; \langle \frac{\nabla z}{z} , Y \rangle_{g} \, \wasdmug
 }
 &&
\\
&  = & \int_{U} z^{2b }   \; \langle \frac{\nabla f}{z} , Y \rangle^{2}_{g} \, \wasdmug
  { { +}}
  \int_{U} z^{2b }   \; \langle \frac{\nabla z}{z} , Y \rangle^2_{g} \, \wasdmug +\int_Uz^{2b}o(1)|Y|^2
\\
&&\underbrace{- \int_{{\partial U}} z^{2b } \; \langle \frac{\nabla z}{z} , Y \rangle_{g}
 \langle \frac{\nabla f}{z} , Y \rangle_{g}\langle \frac{\nabla f}{z} , \eta \rangle_{g} \, \wasdmug }_{=:(*)}
  \,,
\end{eqnarray*}
and note that the contribution from the boundary term equals
\begin{eqnarray*}
&&  (*)=-\int_{{\partial U}} z^{2b } \; \langle \frac{\nabla z}{z} , Y \rangle_{g}
 \langle\eta , Y \rangle_{g} \, \wasdmug
  \,,
\end{eqnarray*}
keeping in mind the sign of $\eta$ as in \eq{21X15.2}.

Next, we integrate the divergence of $\frac 12 \langle \frac{\nabla f}{z} , Y \rangle_{g}
 \langle \frac{\nabla f}{z} , Y \rangle_{g}z^{2b-1}\nabla z$, to obtain
\begin{eqnarray*}
\nn\lefteqn{
    - \int_{U} z^{2b } {\nabla}Y(\frac{\nabla z}{z},\frac{\nabla f}{z}) \; \langle \frac{\nabla f}{z} , Y \rangle_{g} \, \wasdmug
    }
 &&
\\
 &
  =
 &
 \int_{U} z^{2b }  \left(b-\frac{n-1}2   \right) \; \langle \frac{\nabla f}{z} , Y \rangle^2_{g}
 +\int_Uz^{2b}o(1)|Y|^2
\\
 &&
  - \frac12\int_{{\partial U}} z^{2b } \;
 \langle \frac{\nabla f}{z} , Y \rangle _{g}\langle \frac{\nabla z}{z} , \eta \rangle_{g} \, \wasdmug
  \,.
\end{eqnarray*}
Adding both identities, the result follows.
\qedskip

Note that in our applications  the scalar product $ \langle \frac{\nabla z}{z} , \eta \rangle_{g}$ will
be zero on the vertical boundary, and equal to one on the horizontal boundary.

\begin{lemma}{}\label{LpropS2bordsprop2}
For any $Y$ with compact support in  $M$ and for all
$ b \in \R$ it holds that
 \ptcheck{corrected 24 X}
\begin{eqnarray} \label{24X15.3}
\lefteqn{\hspace{1cm}
  -2\int_{U} z^{2b } {S}(Y)(\frac{\nabla f}{z},\frac{\nabla f}{z}) \; \langle \frac{\nabla f}{z} , Y \rangle_{g} \, \wasdmug
  }
  &&
\\
 &&\nn
  =
    2\int_{U} z^{2b }   \; \left(\langle \frac{\nabla z}{z} , Y \rangle_{g}
\langle \frac{\nabla f}{z} , Y \rangle_{g} +o(1)|Y|^2\right) \, \wasdmug   + \int_{{\partial U}} z^{2b } \; \langle \frac{\nabla f}{z} , Y \rangle^{2}_{g} \, \wasdmug
 \,.
  \phantom{xxx}
\end{eqnarray}
\end{lemma}
\proof
Integrate the divergence of
$z^{2b } \; \langle \frac{\nabla f}{z} , Y \rangle^{2}_{g}\frac{\nabla f}z$.
\qed

\begin{lemma}{}\label{LpropS2bordsprop3}
For any $Y$ with compact support in  $M$ and for all
$ b \in \R$ it holds that
 \ptcheck{ 26 X}
%
\begin{eqnarray} \label{24X15.2}
 \lefteqn{\hspace{8mm}
 -2\int_{U} z^{2b } {S}(Y)(\frac{\nabla f}{z},\frac{\nabla z}{z}) \; \langle \frac{\nabla z}{z} , Y \rangle_{g} \, \wasdmug
 -
 \int_{U} z^{2b } {S}(Y)(\frac{\nabla z}{z},\frac{\nabla z}{z}) \; \langle \frac{\nabla f}{z} , Y \rangle_{g} \, \wasdmug
 }
 &&
  \\
   \nn
&=&
   \int_{U} z^{2b }   \; \left((2b-n  )\langle \frac{\nabla z}{z} , Y \rangle_{g}
 \langle \frac{\nabla f}{z} , Y \rangle_{g} +o(1)|Y|^2\right) \, \wasdmug
 \\
 && \nn+ \frac12\int_{{\partial U}} z^{2b } \; \langle \frac{\nabla z}{z} , Y \rangle^{2}_{g} \, \wasdmug
 -\int_{{\partial U}} z^{2b } \; \langle \frac{\nabla f}{z} , Y \rangle_{g}
 \langle \frac{\nabla z}{z} , Y \rangle_{g}\langle \frac{\nabla z}{z} , \eta \rangle_{g} \, \wasdmug
  \,.
  \phantom{xxxxx}
\end{eqnarray}
\end{lemma}
\proof
Integrate the divergence of
$z^{2b } \; \langle \frac{\nabla f}{z} , Y \rangle_{g}
\langle \frac{\nabla z}{z} , Y \rangle_{g} \frac{\nabla z}{z}+\frac12z^{2b } \; \langle \frac{\nabla z}{z} , Y \rangle^{2}_{g}\frac{\nabla f}z$.
\qed

\begin{prop}\label{propKornbande}
Let $b\in\R\setminus\{ \frac{n-1}2\,,   \frac{n+1}2\}$. There exist  $z_0\in(0,z_1)$
and a constant $C$ such that for all $Y$ compactly supported in $M$ we have
\bel{26X15.4}
\int_{\partial U\cap \{z>z_0\}}|Y|^2+\int_{ U\cap \{z>z_0\}}|Y|^2
+\int_U|S(Y)|^2z^{2b}\geq C \int_U|Y|^2z^{2b}
 \,.
\ee
\end{prop}
\proof
We  implement the strategy of Remark~\ref{R26X15.2}.
Without loss of generality we can assume that $|df|_{\tilde g}\leq 1$.

If $b<\frac{n+1}2$,
 adding of \eq{26X15.1} and \eq{24X15.1}
will cancel a  boundary term coming with an undetermined sign. If $b<0$,  the volume integrand dominates the right-hand side of \eq{26X15.4} for $z$ small enough,
taking into account that the $|Y|^2$ term dominates the $Y(f)^2$ term: $\frac{n+1}2-b+b-\frac{n-3}2>0$.
For $0<b \neq\frac{n-1}2$, an addition of  $b/(n-1-2b)$ times
 \eq{26X15.2} to the previously obtained identity
will again lead  to a volume integrand which dominates the right-hand side of \eq{26X15.4} for $z$ small enough, without affecting the boundary terms for small $z$.

If $b>\frac{n+1}2$, the equation resulting from $-\eq{26X15.1}+\frac 12 \eq{24X15.3}+2\eq{24X15.2}$
has a boundary integrand which close to $\{z=0\}$ is manifestly positive:
$$
 \int_{\partial U} z^{2b-2} (Y(f) +Y(z))^2 \ge 0
 \,.
$$
Note that the volume integral in $-\eq{26X15.1}/2$ dominates the weighted norm of $|Y|^2$, and that the cross-terms
$|Y(f)Y(z)| \le \epsilon^2 |Y(f)|^2 + |Y(z)^2/(4\epsilon)$ introduced by $ \frac 12 \eq{24X15.3}+2\eq{24X15.2}$ can be controlled
using  $-\eq{26X15.1}/2$ and \eq{26X15.5} by choosing $\epsilon$ small enough. This leads to the desired estimate. Alternatively,
the addition to  $-\eq{26X15.1}+\frac 12 \eq{24X15.3}+2\eq{24X15.2}$ of a suitable constant times
\eq{26X15.2} preserves positivity of the boundary integrand and leads directly to a volume integrand which dominates the right-hand side of \eq{26X15.4}.
%
\qed
%


\subsection{A horizontal stripe away from the corner}
 \label{KornbandeH}

In this section we work in the setup of Section~\ref{ss8XI15.1}, see the introductory remarks there.

\begin{lemma}{}\label{cordSbord1}
 \ptcheck{18XI15}
For any $Y$ with compact support in $M$ and
for all $a \in \R$,
\begin{eqnarray}\label{20XI15.1}
\lefteqn{
 -\int_{U} x^{2a+1 }( {S}(Y) + \frac12 tr_{g}( {S}(Y))g)(Y,{\nabla x}{}) \, \wasdmug
 }
 &&
 \\ \nn
& =& \int_{U}x^{2a } \Big\{ (\frac{2a+1}{4}+ O(x)) (|dx|^2|Y|_{g}^{2} +2\langle {\nabla x}{} , Y \rangle^{2}_{g} \Big\} \, \wasdmug   \\
\nn&& - \frac{1}{4} \int_{{\partial U}} x^{2a+1 } (|Y|_{g}^{2} \; \langle {\nabla x}{} , \eta \rangle_{g}
+2 \langle Y , {\nabla x}{} \rangle_{g} \; \langle Y , \eta \rangle_{g} )\, \wasdmug
 \,.
\end{eqnarray}
\end{lemma}

\proof
This is a consequence of Proposition \ref{prop:estiPS} with $V=-\nabla x$,
and with $e^{2u}=x^{2a+1}$.
\qed

\begin{lemma}{}\label{cordSbord2}
 \ptcheck{ 20 XI}
 For any $Y$ with compact support in $M$ and
for all $a \in \R$,
\begin{eqnarray}\label{20XI15.2}
\lefteqn{\hspace{1.5cm}
 -2 \int_{U} x^{2a+1 } {S}(Y)({\nabla x}{},{\nabla x}{}) \; \langle {\nabla x}{} , Y \rangle_{g} \, \wasdmug
 }
 &&
\\
\nn&& = \int_{U} x^{2a }  \Big((2a+1)  + O(x) \Big) \; \langle {\nabla x}{} , Y \rangle^{2}_{g} \, \wasdmug   - \int_{{\partial U}} x^{2a+1} \; \langle {\nabla x}{} , Y \rangle^{2}_{g} \langle {\nabla x}{} , \eta \rangle_{g} \, \wasdmug
 \,.
 \phantom{xxxxx}
\end{eqnarray}
\end{lemma}

 \proof
Integrate the divergence of $\langle Y, \nabla x\rangle^2x^{2a+1}\nabla x$.
Note that this is a particular case  of Proposition \ref{prop:estiPS2.2015}.
\qedskip

Let $f$ be a defining function for $\partial U$, such that
$$
|df|_{ g}^2=1, \mbox{ near $\partial U$ and }
$$
$$
\langle \nabla f,\nabla x\rangle=0 \mbox{ on $\partial U$ near } \{x=0\}
\mbox{ and on }\{x=0\}.
$$
 We thus  have
$$
|df|_{g}^2=(1+O(f))\,,\;\;\;\nabla\nabla f=O_g(1)\,,\;\;\;
\langle \nabla f,\nabla x\rangle=O(x)
 \,,
$$
with the unit normal $\eta$ to $\partial U$ equal to $-\nabla f$.
 Also recall that
$$
|dx|_{g}^2=(1+O(x))\,,\;\;\;\nabla\nabla x=O_g(1)
 \,.
$$

\begin{lemma}{}\label{Sdeuxbordpropx}
 \ptcheck{20XI15}
 We have
\begin{eqnarray}\label{20XI15.3}
 \lefteqn{\hspace{1cm}
  - 2\int_{U} x^{2a+1 } S(Y)(\nabla x ,\nabla f ) \; \langle \nabla f  , Y \rangle_{g} \, \wasdmug
-\int_{U} x^{2a+1 } {S}(Y)(\nabla f ,\nabla f ) \; \langle \nabla x  , Y \rangle_{g} \, \wasdmug
}
&&
 \\
 \nn
&  &
  \phantom{xxxxx}
  =
   \int_{U} x^{2a }  \left(\frac{2a+1}2   \right) \; \langle \nabla f  , Y \rangle^2_{g}  \, \wasdmug
+\int_{U} x^{2a }  O(x) | Y |^2_{g} \, \wasdmug
\\
\nn&&
  \phantom{xxxxx = }+ \frac12\int_{{\partial U}} x^{2a+1 } \;
 \langle \nabla f  , Y \rangle^2_{g} \, \wasdmug
-\int_{{\partial U}} x^{2a+1 } \; \langle \nabla x  , Y \rangle_{g}
 \langle\eta , Y \rangle_{g} \, \wasdmug
  \,.
  \phantom{xxxxx}
\end{eqnarray}
\end{lemma}

\proof
We integrate, first, the divergence of $\langle \nabla f  , Y \rangle_{g}
 \langle {\nabla x} , Y \rangle_{g}x^{2a+1}\nabla f$:
\begin{eqnarray*}
 \lefteqn{ - \int_{U} x^{2a+1 } {\nabla}Y(\nabla f ,\nabla x ) \; \langle \nabla f  , Y \rangle_{g} \, \wasdmug
- \int_{U} x^{2a +1} {S}(Y)(\nabla f ,\nabla f ) \; \langle \nabla x  , Y \rangle_{g} \, \wasdmug
}
 &&
 \\
&& = \int_{U} x^{2a }    O(x)   | Y |^{2}_{g} \, \wasdmug
- \int_{{\partial U}} x^{2a+1 } \; \langle \nabla x  , Y \rangle_{g}
 \langle {\nabla f}{} , Y \rangle_{g}\langle {\nabla f}{} , \eta \rangle_{g} \, \wasdmug  \\
&& = \int_{U} x^{2a }  O(x)  | Y |^{2}_{g} \, \wasdmug
 -\int_{{\partial U}} x^{2a+1 } \; \langle \nabla x  , Y \rangle_{g}
 \langle\eta , Y \rangle_{g} \, \wasdmug  .
\end{eqnarray*}
Next, we integrate the divergence of $\langle \nabla f  , Y \rangle^2_{g}x^{2a+1}\nabla x$:
\begin{eqnarray*}
 \lefteqn{
  - \int_{U} x^{2a+1 } {\nabla}Y(\nabla x ,\nabla f ) \; \langle \nabla f  , Y \rangle_{g} \, \wasdmug
  }
  &&
\\
& =&   \int_{U} x^{2a }  \left( \frac{2a+1}2 \right) \; \langle \nabla f  , Y \rangle^2_{g} \;  \, \wasdmug
+\int_{U} x^{2a }  O(x) | Y |^2_{g} \;  \, \wasdmug
\\
&&- \frac12\int_{{\partial U}} x^{2a+1 } \;
 \langle \nabla f  , Y \rangle^2_{g}\langle \nabla f  , \eta \rangle_{g} \, \wasdmug  .
\end{eqnarray*}
We conclude by adding the equations above.
\qedskip

\begin{lemma}{}\label{Sdeuxbordpropx2}
For any $Y$ with compact support in $M$ and any $a\in\R$,
 \ptcheck{20XI}
\begin{eqnarray}\label{20XI15.4}
&&
\\
\nn
\lefteqn{
 - 2\int_{U} x^{2a+1 } S(Y)(\nabla x ,\nabla f ) \; \langle \nabla x  , Y \rangle_{g} \, \wasdmug
- \int_{U} x^{2a+1 } {S}(Y)(\nabla x ,\nabla x ) \; \langle \nabla f  , Y \rangle_{g} \, \wasdmug
}
&&
\\
\nn
& = &
 \int_{U} x^{2a }  \left({2a+1}   \right) \; \langle \nabla f  , Y \rangle\langle \nabla x  , Y \rangle \;  \, \wasdmug
+\int_{U} x^{2a }  O(x) | Y |^2_{g} \, \wasdmug
\\
\nn&&+ \frac12\int_{{\partial U}} x^{2a+1 } \;
 \langle \nabla x  , Y \rangle^2_{g} \, \wasdmug
-\int_{{\partial U}} x^{2a+1 } \; \langle \nabla x  , Y \rangle_{g}
 \langle\nabla f, Y \rangle_{g} \langle\nabla x, \eta \rangle_{g}\, \wasdmug
  \,.
   \phantom{xxx}
\end{eqnarray}
\end{lemma}

\proof
We integrate the divergence of
$$\langle \nabla f  , Y \rangle_{g}
 \langle {\nabla x} , Y \rangle_{g}x^{2a+1}\nabla x
+\frac12\langle {\nabla x} , Y \rangle^2_{g}x^{2a+1}\nabla f.$$
\qed

\begin{lemma}{}\label{cordSbord2f}
 \ptcheck{20XI}
 For any $Y$ with compact support in  $M$ and
$\forall a \in \R$,
\begin{eqnarray}\label{20XI15.5}
 \lefteqn{
 -2 \int_{U} x^{2a+1 } {S}(Y)({\nabla f}{},{\nabla f}{}) \; \langle {\nabla f}{} , Y \rangle_{g} \, \wasdmug
 }
 &&
\\
\nn&& = \int_{U} x^{2a }   O(x)  \; |Y |^{2}_{g} \, \wasdmug   + \int_{{\partial U}} x^{2a+1} \; \langle {\nabla f}{} , Y \rangle^{2}_{g}  \, \wasdmug
\end{eqnarray}
\end{lemma}
 \proof
Integrate the divergence of $\langle Y, \nabla f\rangle^2x^{2a+1}\nabla f$.
\qed

\begin{prop}\label{propKornbandex}
For any $a>  -1/2$ 
there exist   $x_0\in(0,x_1)$
and a positive constant $C$ such that for all $Y$ compactly supported in $M$,
\bel{16VI16.3}
\int_{\partial U\cap \{x>x_0\}}|Y|^2+\int_{ U\cap \{x>x_0\}}|Y|^2
+\int_U|S(Y)|^2x^{2a+2}\geq C \int_U|Y|^2x^{2a}
 \,.
\ee
 \ptcheck{20XI}
\end{prop}

\proof
Consider the linear combination
\bel{16VI16.2}
 \mbox{2$\times $\eq{20XI15.1} + 4$\times $\eq{20XI15.2}+\eq{20XI15.3}+2$\times$\eq{20XI15.4}+$\frac12\times $\eq{20XI15.5}
}
\ee
of the equations above. One obtains a volume integral involving $S(Y)$ plus
%
\bean
 \int _U \Big(\frac{2a+1}{2}+o(1)\Big)x^{2a} \Big(|Y|^2+6Y(x)^2+(2Y(x)+Y(f))^2\Big)
\\
   + \int_{{\partial U}} x^{2a+1} (Y(f)+Y(x))^2 + \int_{{\partial U\cap \{x>x_0\}}} O(|Y|^2)
    \,,
\eeal{16VI16.2+}
%
%
with some constant $c$.
For $a>-1/2$ the result follows as explained in Remark~\ref{26X15.2}, since all terms in \eq{16VI16.2+} except the last one are positive.
%
%
\qed

\subsection{The global inequality}
 \label{ss13X15.5}

In what follows we work in the framework of Section~\ref{ss18IX15.2}, in particular the set $\Omega$ is as defined at the beginning of that section. The inequalities proved so far lead to:

\begin{Theorem}
  \label{TC22V15}
 Let $\phi= x/\rho$, $\psi=x^{a-1} z^b \rho^{c+1}$, and suppose that $F \subset \mathring H^1_{\phi,\psi}(\Omega)$ is a
   closed subspace of $ \mathring H^1_{\phi,\psi}(\Omega)$ transverse to the kernel of $S$. Then
 for all  $b\neq (n+1)/2, (n-1)/2$, $c>-|n-1-2b|$ and for all constants $a$ sufficiently large there exists a  constant $C_1(a,b,c, F)>0$ such that the inequality
\bel{23V15.2}
	\|\phi S(Y)\|_{L^2_\psi}\geq C_1(a,b,c,F) \|Y\|_{  H^1_{\phi,\psi}}
\ee
holds for all $Y\in F$.
\end{Theorem}

\begin{remark}{\rm
  \label{R23V15.1}
We can take  $F = \mathring H^1_{\phi,\psi}(\Omega)$ when $(\Omega,g)$ has no non-trivial Killing vector fields in $\mathring H^1_{\phi,\psi}(\Omega)$
\myqed
}\end{remark}

{\noindent\sc Proof of Theorem~\ref{TC22V15}:}
Let $\chi_1:\R\to \R^+$ be a smooth function such that $\chi_1(x)=1$ for $0\le x\le x(a,b,c)/2$, and $\chi_1(x)=0$ for $ x\ge  x(a,b,c) $, set $\psi_1 = 1 - \chi_1$. Similarly let $\chi_2:\R\to \R^+$ be a smooth function such that $\chi_2(z)=1$ for $0\le z\le z(a,b,c)/2$, and $\chi_2(z)=0$ for $  z\ge  z(a,b,c) $, set $\psi_2 = 1 - \chi_2$.
We have
$$
 Y =  (\chi_1  +\psi_1 ) (\chi_2    +\psi_2 ) Y =
   \chi_1   \chi_2 Y +\chi_1   \psi_2 Y + \psi_1  \chi_2 Y+ \psi_1  \psi_2 Y
    \,.
$$
Since $\chi_1 \chi_2 Y $ has support in the set $\{0<x<x(a,b,c)\,,\, 0<z<z(a,b,c)\} $, Proposition~\ref{prop:estiS2ah2015} applies and gives
\beal{22V15.22}
 \lefteqn{\hspace{1cm}
\intOmega
x^{2a-2}z^{2b}\rho^{2c+2}|\chi_1 \chi_2 Y|^2\;\wasdmug
 }
 &&
\\
 \nonumber
  &&
 \leq
   C(a,b,c)^{-1}
    \intOmega  x^{2a}z^{2b}\rho^{2c}|S(\chi_1 \chi_2 Y)|^2\; \wasdmug
\\
 \nonumber
  &&
 =
   C(a,b,c)^{-1}
    \intOmega  x^{2a}z^{2b}\rho^{2c}
     |\frac 12 \nabla (\chi_1 \chi_2) \otimes Y
     + \frac 12 Y\otimes \nabla (\chi_1 \chi_2)
    + \chi_1 \chi_2 S( Y)|^2\; \wasdmug
\\
 \nn &&
 \leq
   C'  \left(
    \intOmega  x^{2a}z^{2b}\rho^{2c}|S( Y)|^2\; \wasdmug
    +  \int_{\{\nabla (\chi_1 \chi_2) \ne 0\}}
		x^{2a}z^{2b}\rho^{2c}|Y|^2\; \wasdmug
    \right)
 \,.
\eea

Since  $x<x(a,b,c) $  on $\{\chi_1>0\}$ and  $z(a,b,c)/2<z \le C_z$ on $\{\psi_2>0\}$ for some constant $C_z$,
using~\cite[Proposition~5.1]{ChDelay} in the second inequality below one obtains
\beal{22V15.24}
 \lefteqn{
\intOmega
x^{2a-2}z^{2b}\rho^{2c+2}|\chi_1 \psi_2 Y|^2\;\wasdmug  \le
C \intOmega
x^{2a-2} |\chi_1 \psi_2 Y|^2\;\wasdmug
 }
 &&
\\
 \nonumber
 &&
 \leq
   C''
    \intOmega  x^{2a} |S(\chi_1 \psi_2 Y)|^2\; \wasdmug
\\
 \nonumber
  &&
 \leq
   C'''  \left(
    \int _{\{\chi_1\psi_2\ne 0\}} x^{2a} |S( Y)|^2\; \wasdmug
    +  \int_{\{\nabla (\chi_1 \psi_2) \ne 0\}}  |Y|^2\; \wasdmug
    \right)
\\
 \nn &&
 \leq
   C''''  \left(
    \intOmega  x^{2a}z^{2b}\rho^{2c}|S( Y)|^2\; \wasdmug
    +  \int_{\{\nabla (\chi_1 \psi_2) \ne 0\}}  |Y|^2\; \wasdmug
    \right)
 \,.
\eea
Since $ z< z(a,b,c) $ on $\{\chi_2>0\}$  and $x(a,b,c)/2<x\le C_x$
on $\{\psi_1>0\}$ for some constant $C_x $, using~\cite[Proposition~6.1]{ChDelay}
in the second inequality below one obtains, with some constants possibly different from the ones of the previous calculation,
\beal{22V15.25-}
 \lefteqn{
\intOmega
x^{2a-2}z^{2b}\rho^{2c+2}|\chi_2 \psi_1 Y|^2\;\wasdmug  \le
C \intOmega
z^{2b} |\chi_2 \psi_1 Y|^2\;\wasdmug
 }
 &&
\\
 \nonumber
 &&
 \leq
   C''
    \intOmega  z^{2b} |S(\chi_2 \psi_1 Y)|^2\; \wasdmug
\\
\nonumber
  &&
 \leq
   C'''  \left(
    \int _{\{\chi_2\psi_1\ne 0\}} z^{2b} |S( Y)|^2\; \wasdmug
    +  \int_{\{\nabla (\chi_2 \psi_1) \ne 0\}}  |Y|^2\; \wasdmug
    \right)
\\
  \nn&&
 \leq
   C''''  \left(
    \intOmega  x^{2a}z^{2b}\rho^{2c}|S( Y)|^2\; \wasdmug
    +  \int_{\{\nabla (\chi_2 \psi_1) \ne 0\}}  |Y|^2\; \wasdmug
    \right)
 \,.
\eea
Adding, we conclude that there exists a constant $C$ such that
\beal{22V15.25}
 \lefteqn{
\intOmega
x^{2a-2}z^{2b}\rho^{2c+2}| Y|^2\;\wasdmug \leq
   C \left(
    \intOmega  x^{2a}z^{2b}\rho^{2c}|S( Y)|^2\; \wasdmug
     \right.
 }
 &&
%
%
\\
 \nn &&
 \left.
    +  \int_{K_1}  |Y|^2+\int_{\{ \chi_1 \nabla\chi_2 \ne 0\}\cup
		\{ \chi_2 \nabla\chi_1 \ne 0\}}
		x^{2a}z^{2b}\rho^{2c}|Y|^2\; \wasdmug
    \right)
 \,,
\eea
where $K_1$ is the compact set defined as the closure of
\bel{22V15.26}
  \{ \psi_2 \nabla  \chi_1  \ne 0\}\cup
	\{\psi_1 \nabla  \chi_2  \ne 0\}\cup \{   \psi_1 \psi_2 \ne 0\}
 \,.
\ee

Now, the non-compact set where $\{ \chi_1 \nabla\chi_2 \ne 0\}$ or
$\{ \chi_2 \nabla\chi_1 \ne 0\}$ is a subset
of the union of a vertical type domain (see Section \ref{KornbandeV}) where
$x$ and $\rho$ are bounded above and below by two positive constants, and a  horizontal one (see Section~\ref{KornbandeH}) where
$z$ and $\rho$ are bounded above and below by two positive constants.
We can then use Propositions \ref{propKornbande} and \ref{propKornbandex}
to control the weighted $L^2$-norm of $Y$ on this set in terms of a weighted norm of $S(Y)$.
Hence, there exists a compact  set $K_2$, a relatively compact hypersurface $\Sigma$, and a constant $C_2$ such that
\bea\label{26X15.11}
 \lefteqn{
\int_{\{ \chi_1 \nabla\chi_2 \ne 0\}\cup
		\{ \chi_2 \nabla\chi_1 \ne 0\}\}}
		x^{2a}z^{2b}\rho^{2c}|Y|^2
}
&&
 \\
		\nn&&
    \leq C_2\left(
		 \int_\Omega x^{2a}z^{2b}\rho^{2c}|S( Y)|^2\;
    +  \int_{K_2}  |Y|^2+  \int_{\Sigma}  |Y|^2\right)
     \,.
\eea
Adding \eq{22V15.25} and \eq{26X15.11},  keeping in mind that $x^2/\rho^2\leq 1$,
we find that there exists a constant $C$ such that
\bel{22V15.25b}
\intOmega
x^{2a-2}z^{2b}\rho^{2c+2}| Y|^2\;\wasdmug
 \leq
   C \left(
    \intOmega  x^{2a}z^{2b}\rho^{2c}|S( Y)|^2\; \wasdmug
    +  \int_{K}  |Y|^2+  \int_{\Sigma}  |Y|^2\; \wasdmug
    \right)
 \,,
\ee
where $K=K_1\cup K_2$.

As such, for any smooth vector fields $V$ with compact
support in $\Omega$ it holds that
$$
2\int_\Omega|S(V)|^2=\int_\Omega|\nabla V|^2+(\mbox{div} \;V)^2-\Ric(V,V)\geq
\int_\Omega|\nabla V|^2-C_3|V|^2
 \,.
$$
Choosing $V=\phi\psi Y=x^{a}z^{b}\rho^{c} Y$, and using
(\ref{22V15.25b}), one deduces that
\bel{estiSH1}
    \| Y\|^2_{H^1_{\phi,\psi}(\Omega)}\;
    \leq \;
        C
        \left(
        \int_\Omega  x^{2a}z^{2b}\rho^{2c}|S( Y)|^2\; \wasdmug
        +  \int_{K}  |Y|^2+  \int_{\Sigma}  |Y|^2\; \wasdmug
    \right)
 \,.
\ee
The usual contradiction argument, which invokes compactness of the embeddings $H^1(K)\subset L^2(K)$ and $H^1(K)\subset H^{1/2}(K)\subset  L^2(\Sigma)$
(
compare~\cite[Proposition~3.1 and Remark~3.2]{ChDelay}) concludes the proof.
\qedskip

We continue with an equivalent of~\cite[Proposition~D.13]{ChDelay}.

\begin{prop}\label{prop:estiN2015}
For all  $b\neq (n+1)/2, (n-1)/2, (n-3)/2$, $c>-|n-1-2b|$, and for all $a$ sufficiently large there exist constants
$C(a,b,c)>0$,
$x(a,b,c)>0$ and $z(a,b,c)>0$ such that for all differentiable function $N$ with
compact support in $\{0<x<x(a,b,c)\}\cap \{0<z<z(a,b,c)\}$ we have
\beal{23V15.1}
 \lefteqn{
 \intOmega  x^{2a}z^{2b}\rho^{2c}|\nabla\nabla N-\Delta N g-N \Ric (g)|^2\;\wasdmug
 }
 &&
\\
 \nn&&
   \geq
 C(a,b,c)\intOmega  x^{2a-2}z^{2b}\rho^{2c+2}( {x^{-2}}{\rho^2}|N|^2+|\nabla N|^2)\;\wasdmug \,.
\eea
\end{prop}

Note that since the operator involved is of order two, the natural weight functions for the inequality \eq{23V15.1} are
$$
 \phi = \frac x \rho
 \,,
 \quad
  \psi = x^{a-2}z^b \rho^{c+2}
  \,.
$$

\begin{proof}
We will use Proposition \ref{prop:estiS2ah2015} with
$$ Y=z^{-1}\nabla N-N\nabla(z^{-1})=z^{-2}\nabla(zN) \,.$$
We have
$$
\nabla_{(i}Y_{j)}=z^{-1}\nabla_i\nabla_j
N-N\nabla_i\nabla_j(z^{-1})=z^{-1}[(\nabla_i\nabla_j
N-Ng)+NO_g(z)]\,,
$$
then
\bel{23V15.5}
S(Y)-\mathrm{div}Y g=
z^{-1}\Big[\nabla\nabla N-\Delta N
g+(n-1)Ng+O_g(z)N\Big]
 \,.
\ee
Using
\begin{eqnarray*}
\Ric (g) &=&-(n-1){g}+O_g(z)
 \,,
\end{eqnarray*}
we can rewrite \eq{23V15.5} as
$$
 S(Y)-\mathrm{div}Y g=
 z^{-1}[\nabla\nabla N-\Delta N g-N(\Ric (g)+O_g(z))]
 \,.
$$
Now, we use the inequality
$$
|S(Y)-\tr S(Y) g|^2=|S(Y)|^2+(n-2)(\tr S(Y))^2\geq|S(Y)|^2,
$$
and Proposition \ref{prop:estiS2ah2015} with $b=b_{\mathrm{there}}$  replaced by
$b_{\mathrm{there}}=b_{\mathrm{here}}+1=b+1$ yields
\beal{20IX15.23}
 \lefteqn{
\int x^{2a}z^{2b}\rho^{2c}\Big(|\nabla\nabla N-\Delta N g-N\Ric (g)|^2
+O(z^2)N^2\Big) }
 &&
\\
 \nn&&
 \phantom{xxxxxxxxxxxxxxxxxx}
 \ge C \int x^{2a-2}z^{2b-2 }\rho^{2c+2} |\nabla (zN)|^2
 \,.
\eea
For further reference, we note that if $N$ is not compactly supported near the boundary, using \eq{22V15.25}
instead of Proposition~\ref{prop:estiS2ah2015} we will obtain
\beal{20IX15.23b}
 \lefteqn{
\int x^{2a}z^{2b}\rho^{2c}\Big(|\nabla\nabla N-\Delta N g-N\Ric (g)|^2 \Big)
+ \int x^{2a}z^{2b+2}\rho^{2c}\,N^2
  }
 &&
\\
 \nn&&
 \phantom{xxx}
 + \|N\|^2_{H^1(K)}
 \ge C \int x^{2a-2}z^{2b-2 }\rho^{2c+2} |\nabla (zN)|^2
 \,.
\eea

Returning to \eq{20IX15.23}, the right-hand side can be estimated from below as follows,  where
Proposition~\ref{P21VI14.1n}  with $\bmu$ there replaced by $b-1$ and $u=zN$ is
used when going from the third to the fourth line:
%
\beal{20IX15.24}
 \lefteqn{\hspace{1cm}
 \|x^{a}\rho^{c}z^{b-1}\nabla(zN)\|_{L^2}
 }
 &&
\\
 \nn
  & &=
  \epsilon\|x^{a}\rho^{c}z^{b-1}\nabla(zN)\|_{L^2}+(1-\epsilon)\|x^{a}\rho^{c}z^{b-1}\nabla(zN)\|_{L^2}
\\
 \nn
 &&\geq
  \epsilon\|x^{a}\rho^{c}z^{b}\nabla N\|_{L^2}-\epsilon\|x^{a}\rho^{c}z^{b-1}N\nabla z\|_{L^2}+(1-\epsilon)\|x^{a}\rho^{c}z^{b-1}\nabla(zN)\|_{L^2}
\\
 \nn
 &&\geq
  \epsilon\|x^{a}\rho^{c}z^{b}\nabla N\|_{L^2}-\epsilon\|x^{a}\rho^{c}z^{b-1}N\nabla z\|_{L^2}+(1-\epsilon)c\|x^{a-1}\rho^{c+1}z^bN\|_{L^2}
\\
 \nn& &\geq
 \epsilon
  C^{-1}
(\|x^{a}\rho^{c}z^{b}\nabla N \|_{L^2}+\|x^{a-1}\rho^{c+1}z^bN\|_{L^2})
  \,,
\eea
with $\epsilon>0$ small so that
$$
\epsilon\|x^{a}\rho^{c}z^{b-1}N\nabla z\|_{L^2}\le \frac{(1-\epsilon)c}{2}\|x^{a-1}\rho^{c+1}z^bN\|_{L^2}
 \,.
$$

The last term on the left-hand side of \eq{20IX15.23}
can be absorbed in the last      term in \eq{20IX15.24} when either of $x$ or $z$ is small enough on the support of $N$.
\end{proof}

\begin{corollary}
  \label{C23V15}
	Let $\phi= x/\rho$, $\psi=x^{a-2} z^b \rho^{c+2}$, and suppose that $\mathcal F \subset \mathring H^2_{\phi,\psi}$ is a
   closed subspace of $ \mathring H^2_{\phi,\psi}$ transverse to the kernel of $P^*_g$. Then
 for all  $b\neq (n+1)/2, (n-1)/2, (n-3)/2$, $c>-|n-1-2b|$ and for all $a$ sufficiently large there exists a  constant $C_2(a,b,c, \mathcal F)>0$ such that for all $N\in \mathcal F$ we have
\bel{3X15.2}
	\|
         \phi^2\big(\nabla\nabla N-\Delta N g-N \Ric (g)\big)
            \|_{L^2_\psi}
                 \geq C(a,b,c,\mathcal F)
        \|N\|_{\mathring H^2_{\phi,\psi}}
 \,.
\ee
\end{corollary}

\begin{proof}
We return to the argument of Proposition~\ref{prop:estiN2015}. The calculations which follow immediately after
\eq{20IX15.23b} are instead carried-out as follows:
\beal{20IX15.24b}
 \lefteqn{\hspace{1cm}
 \|x^{a}\rho^{c}z^{b-1}\nabla(zN)\|_{L^2}
 }
 &&
\\
 \nn
 &\geq &
  \epsilon\|x^{a}\rho^{c}z^{b}\nabla N\|_{L^2}-\epsilon\|x^{a}\rho^{c}z^{b-1}N\nabla z\|_{L^2}
+(1-\epsilon)\|x^{a}\rho^{c}z^{b-1}\nabla(zN)\|_{L^2}
\\
 \nn
 &\geq &
  \epsilon\|x^{a}\rho^{c}z^{b}\nabla N\|_{L^2}-\epsilon\|x^{a}\rho^{c}z^{b-1}N\nabla z\|_{L^2}
 \\
 \nn
 &&
 +(1-\epsilon)c\big(\|x^{a-1}\rho^{c+1}z^bN\|_{L^2} - \|N\|_{L^2(K)}
 \big)
\\
 \nn&\geq &
 \epsilon
  C^{-1}
(\|x^{a}\rho^{c}z^{b}\nabla N \|_{L^2}+\|x^{a-1}\rho^{c+1}z^bN\|_{L^2})
 -(1-\epsilon) \|N\|_{L^2(K)}
  \,,
\eea
where $K$ is a suitable compact set as in \eqref{inegapoincareglobalmodK}, for all $\epsilon$ small enough.

Next, we write
\bean
 \nn\lefteqn{
 \int x^{2a}z^{2b+2}\rho^{2c}\,N^2 \le
\int_{z\le \epsilon} x^{2a-2}z^{2b}\rho^{2c+2}\, \frac{x^2 z^2}{\rho^2}N^2
 }
 &&
\\
\nn
 &&
 +
\int_{x\le \epsilon} x^{2a-2}z^{2b}\rho^{2c+2}\, \frac{x^2 z^2}{\rho^2}N^2 +
\int_{x\ge \epsilon\,,\ z \ge \epsilon} x^{2a-2}z^{2b}\rho^{2c+2}\, \frac{x^2 z^2}{\rho^2}N^2
\\
 \nn
 &&
 \le
\epsilon^2 \int_{z\le \epsilon} x^{2a-2}z^{2b}\rho^{2c+2}\, N^2 +
 \epsilon^2 \int_{x\le \epsilon} x^{2a-2}z^{2b}\rho^{2c+2}\, N^2
  + C \int_{K} N^2
\\
 \nn
 &&
 \le
2 \epsilon^2 \int x^{2a-2}z^{2b}\rho^{2c+2}\, N^2
  + C \int_{K} N^2
 \,,
\eeal{3X15.5}
making the compact set $K=K(\epsilon)$ larger if necessary. Inserting all this into \eq{20IX15.23b} we obtain, for $\epsilon$ small and after some rearrangements,
\bel{3X15.21}
	\|
         \phi^2\big(\nabla\nabla N-\Delta N g-N \Ric (g)\big)
            \|_{L^2_\psi}+\|N\|_{\mathring H^1_{\phi,\psi}(K)}
                 \geq C(a,b,c)
        \|N\|_{\mathring H^1_{\phi,\psi}}
 \,.
\ee
The end of the proof is the usual contradiction argument.
\end{proof}

\appendix

\section{No KIDs}
 \label{s2X15.1}

Let $h=\delta g$ and $Q=\delta K$, the linearisation $ P_{(K,g)}$
of the constraints map at $(K,g)$ reads \be \label{2}
P_{(K,g)}(Q,h)= \left(
\begin{array}{l}
-K^{pq}\nabla_i h_{pq}+K^q{}_i(2\nabla^j h_{qj}-\nabla_q h^l{}_l)\\
\;\;\;\;\;\;\;-2\nabla^jQ_{ij}+2\nabla_i\;\tr  Q
-2(\nabla_iK^{pq}-\nabla^qK^p{}_i)h_{pq}\\
  \\
-\Delta(\tr  h)+\divr    \divr    h-\langle h,\Ricc(g)\rangle +2K^{pl}K^q{}_l h_{pq}\\
\;\;\;\;\;\;\;-2\langle K,Q\rangle +2\tr  K(-\langle h,K\rangle
+\tr Q)
\end{array}
\right)\,. \ee

Recall that a KID is defined as a solution $(N,Y)$ of the set of
equations $P_{(K,g)}^*(Y,N)=0$, where $P_{(K,g)}^*$ is the formal
adjoint of $P_{(K,g)}$:
\bea
 \label{4}
&&
\\
 \nn
  \lefteqn{
   P_{(K,g)}^*(Y,N)
   =
   }
   &&
\\
 \nn
  &&
   \left(
\begin{array}{l}
2(\nabla_{(i}Y_{j)}-\nabla^lY_l g_{ij}-K_{ij}N+\tr K\; N g_{ij})\\
 \\
\nabla^lY_l K_{ij}-2K^l{}_{(i}\nabla_{j)}Y_l+
K^q{}_l\nabla_qY^lg_{ij}-\Delta N g_{ij}+\nabla_i\nabla_j N\\
\; +(\nabla^{p}K_{lp}g_{ij}-\nabla_lK_{ij})Y^l-N \Ricc(g)_{ij}
+2NK^l{}_iK_{jl}-2N \tr_gK K_{ij}
\end{array}
\right)
 \,.
\eea
We denote by $\mcK(\Omega)$ the set of KIDs
defined  on an open set $\Omega$.

In order to analyze the set of KIDs in an asymptotically hyperbolic setting, it is convenient to rewrite the KID equations in the following equivalent form
\bea
 \label{2X15.4}
 \nabla_{(i}Y_{j)}
  & = & K_{ij}N
 \,,
\eea
\beal{2X15.2}\;\;\;\;
 \nabla_i \nabla_j N
  & =& \big(
   \Ricc(g)_{ij}
    - 2K^l{}_iK_{jl} +\tr_gK K_{ij}
    - K^{ql}K_{ql}g_{ij}\big)
        N
     + \Delta N g_{ij}
\\
 \nn&&
  -(\nabla^{p}K_{lp}g_{ij}-\nabla_lK_{ij})Y^l
   + 2K^l{}_{(i}\nabla_{j)}Y_l
 \,.
\eea
Taking traces, we obtain
\be
 \Delta N
  = -\frac{1}{n-1}\bigg(
   \big(
    R
    +(\tr_gK )^2
    - n K^{ql}K_{ql} \big)
        N
  -\big(n \nabla^{p}K_{lp} -\nabla_l\tr_gK\big)Y^l \bigg)
 \,,
  \label{2X15.3}
\ee
which allows one to eliminate the second derivatives of $N$ from the right-hand side of \eq{2X15.2}, leading to
%
%
\beal{2X15.5}
\lefteqn{\hspace{1cm}
 \nabla_i \nabla_j N
  =
  }
 &&
\\
\nn
 &&
   \bigg(
   \Ricc(g)_{ij}
    - 2K^l{}_iK_{jl} +\tr_gK K_{ij}
     + \frac{1}{ 1-n }
   \big(
    R
    +(\tr_gK )^2
    -  K^{ql}K_{ql} \big) g_{ij}
 \bigg)
        N
\\
 \nn&&
  +\big(\nabla_lK_{ij}+ \frac{1}{n-1}(\nabla^{p}K_{lp}- \nabla_l \tr_gK)g_{ij}
   \big)Y^l
   + 2K^l{}_{(i}\nabla_{j)}Y_l
 \,.
\eea

Let us show  that metrics that are $C^{1,\alpha}$-compactifiable for some $\alpha\in (0,1]$ (equivalently, $z^2 g$ is $C^{1,\alpha}$-extendible across $\{z=0\}$), and whose derivatives up to order three satisfy a weighted condition,
have no non-trivial static KIDs which are $o_g(1/z)$. Further, initial data sets with such $g$'s and for which $K$ is proportional to the metric to leading order, with a constant proportionality factor, and whose derivatives up to order two satisfy a weighted condition, have no nontrivial KIDs.  Indeed, we have:

\medskip

\begin{Proposition}
  \label{P2X15.1}
  Let $\alpha>0$, $\tau \in \R$, and consider a metric $g\in M_{\mathring g+ C^3_{1,z^{-1-\alpha}}}$ (see Definition~\ref{D21IX15.1}), where $\zg$ is $C^3$-compactifiable. Suppose that $K-\mathring K -\tau g \in {C^2_{1,z^{-1-\alpha}}}$, where
  $$
   \mathring K \in C^2_{1,z^{-1}}
    \,.
  $$
Let $(N,Y)$ be in the kernel of $P^*_{(K,g)}$ and suppose that there exists $\lambda>0$ such that $N,Y\in C^2_{1,z^{-\lambda+1}}$. Then $(N,Y)\equiv 0$.
\end{Proposition}

\begin{remark}{\rm
 \label{R4X15.1}
 Initial data such that $g$ is $C^3$-compactifiable and $z^2 K$ is $C^2$ up-to-the conformal boundary satisfy our differentiability hypotheses.
\myqed
}\end{remark}

\begin{remark}{\rm
 \label{R4X15.2}
 The argument below can be used to derive an asymptotic expansion for non-vanishing KIDs, but this is of no concern to us here.
\myqed
}\end{remark}

\proof
Define
$$
 I:=\{\rho\,:\ \exists \, C>0 \ \mbox{such that}\ |N|+|Y|_g+|\nabla N|_g+|\nabla Y|_g\le C z^\rho
  \ \mbox{for small $z$}\}
  \,.
$$
By hypothesis $I$ is non-empty. Setting
$$
  \hat \rho := \sup I
 \,,
$$
it holds that    $\hat \rho \ge  \lambda-1 >-1$.

Without loss of generality, decreasing $\alpha$ if necessary we can assume that
$$
 0<\alpha\le 1
 \,.
$$

We wish, first, to show that $\hat \rho = \infty$. Suppose, for contradiction, that $\hat \rho<\infty$. Let $\rho\in (\hat \rho-\alpha/2, \hat \rho)$; replacing $\rho$ by a slightly larger number if necessary we can assume that
\bel{4X15.5}
 \mbox{$\rho\in (\hat \rho-\alpha/2, \hat \rho)$, $\rho+\alpha\ne 0$,  $\rho+\alpha\ne 1$,  $\rho+\alpha/2\ne 0$,  $\rho+\alpha/2 \ne 1$, $\rho>-1$.}
\ee
\Eq{2X15.5} yields
\bel{4X15.11}
 |\nabla \nabla N|_g = O(z^\rho)
\,.
\ee
A standard calculation using
\eq{2X15.4} gives
\bel{2X15.16}
 \nabla_i \nabla_j Y_k = R_{\ell i j k} Y^\ell
 - \nabla_k (N K_{ij} )
 + \nabla_i (N K_{kj} )
 + \nabla_j (N K_{ik} )
 \,,
\ee
which implies
\bel{4X15.11a}
 |\nabla \nabla Y|_g = O(z^\rho)
\,.
\ee

There exists a coordinate system (see~\cite[Appendix~B]{AndChDiss}) in which $g$ can be written as
$$
 g = z^{-2}\big((1+O(z^{1+\alpha}))dz^2 + h_{ab} d\theta^a d\theta^b + O(z^{1+\alpha})_{a} d\theta^a dz
    \big)
  \,.
$$
Setting $ \mathring h_{ab}= h_{ab}|_{z=0}$ and $ \mathring h'_{ab}= \partial_zh_{ab}|_{z=0}$, it holds that
\begin{eqnarray}
 \label{christ_zg1}
 &
  \Gamma^c_{ab} =  %
 \Gamma[{\ringh}{}]^c_{ab} +
  O(z^{\alpha})
 \,,
\quad
 \Gamma^z_{ab}= z^{-1} {\ringh}{}_{ab} +\frac12{\ringh}{}'_{ab}+O(z^{\alpha})
\,,
\\
 &
 \Gamma^z_{za}  =    O(z^{\alpha})
\,,
 \quad
\label{christ_zg2}
     \Gamma^z_{zz}
      =   -z^{-1} +O(z^{\alpha})\,,
      &
\\
 &
 \Gamma^c_{za}= -z^{-1} \delta_a^c - \frac12\ringh^{cd}{\ringh}'_{da}+O(z^{\alpha})
\,,
\quad
 \Gamma^c_{zz}   = O(z^{\alpha})
\,.
 &
\end{eqnarray}
We can calculate $\nabla_z\nabla_iN$,   and $\nabla_z Y_i$ using \eq{2X15.4} and \eq{2X15.5},
obtaining thus
\beal{2X15.6}
 \big(z^2 \partial^2_z + z \partial_z - 1\big)N & = &  O(z^{\alpha+\rho })=:\psi
  \,,
\\
 \label{2X15.7}
 \partial_z\big(  \partial_a (z N )\big)+ \frac 12\ringh^{cd}{\ringh}'_{da}\partial_c (zN) & = &  O(z^{\alpha+\rho })=:\psi_a
  \,,
\\
 \label{2X15.8}
 z\partial_z(z Y_z) & = &  N \, O(1)+  O(z^{\alpha+\rho }):= \lambda
  \,,
\\
 \label{2X15.9}
 \partial_z(z^2 Y_a) + \ringh^{cd}{\ringh}'_{da}(z^2Y_c)& = & -
 \partial_a(z^2 Y_z) + \underbrace{O(z^{\alpha+\rho +1})}_{=: \lambda_a}
  \,.
\eea
Scaling the variable $z\in[0,z_0]$ if necessary, we can without of generality assume that $z_0\ge 1$.
\Eq{2X15.7} can be thought of as a system of ODEs of the form
\bel{8XI15.1}
 \partial_zZ=AZ+\psi
 \,,
\ee
with $A(\theta)=-\frac12\mathring h'\mathring h^{-1}$ and
$
 Z_a:=\partial_a(zN)
$.
We define the matrix
$$
W(z,\theta):=\exp(zA(\theta)).
$$
The solution of \eq{8XI15.1} can be explicitly written as
\bel{1XI15.1}
 W^{-1}(z,\theta)Z(z,\theta) = -\int_z^1 \left(W\psi\right)(s,\theta) ds +  W^{-1}(1,\theta) Z(1,\theta)
 \ee
Since $\ringh_{ab}$ is $C^3$, the matrix  $W(z,\theta)$  is twice-differentiable up-to-$\{z=0\}$ in all variables.
In particular the limit $z\to0$ of the right-hand side of \eq{1XI15.1}, and hence of $(W^{-1}Z)_a$, exists and we have
\beal{1XI15.2}
  (W^{-1}Z)_a(0,\theta)
  &= &
     -\int_0^1 \left(W_{a}{}^b \psi_b\right)(s,\theta) ds  +  (W^{-1}Z)_a(1,\theta)
 \,,
\\
  (W^{-1}Z)_a(z,\theta)
 \nn &= &  (W^{-1}Z)_a(0,\theta)+
 \underbrace{ \int_0^z \left(W_{a}{}^b \psi_b\right)(s,\theta) ds}_{=:(W^{-1}){}_a{}^b\chi_b(z,\theta)= O(z^{\alpha+\rho +1})}
 \,.
\eea
We conclude that there exist differentiable functions $C_a(\theta):= Z_a(0,\theta)$ such that
$$
z\partial_a N =C_a(\theta)+ O(z^{\alpha+\rho +1})
 \,.
$$
Given a point of $\partial M$ with local coordinates $
\theta_0$,
integration in $\theta$ near $\theta_0$ gives
\beaa
 N(z,\theta)
   &= &
     \underbrace{N(z,\theta_0)}_{ O(z^{ \rho }) } + z^{-1}\int_{t=0}^1  C_a(t \theta+(1-t)\theta_0) (\theta^a-\theta_0^a)\, dt
\\
 &&
       +
  \underbrace{z^{-1}\int_{t=0}^1 \chi_a(z, t \theta+(1-t)\theta_0)(\theta^a-\theta_0^a)\, dt}_{ O(z^{\alpha+\rho }) }
 \,.
\eeaa
Taking into account the condition $N=o(1/z)$  gives   $C_a\equiv0$, so that
there exists a function $f(z)=O(z^\rho)$  such that
\bel{16X15.1}
N = f(z)  + O(z^{\alpha+\rho })
 \,,
  \quad
 z\partial_a N =  O(z^{\alpha+\rho  })
 \,.
\ee

Assume, first, that $\alpha+\rho< 1$.
Integrating \eq{2X15.6}, there exists a function $A(\theta)$ such that
$$
 N = A z -  \frac z2 \int_z^1  \frac{\psi(s)}{s^2}ds
 -  \frac 1{2z}  \int_0^z  {\psi(s)} ds =  O(z^{\alpha+\rho})
  \,.
$$
(A general solution of  \eq{2X15.6}  would have a supplementary term $B(\theta)/z$, which must be zero by the boundary conditions.)
Further
$$
 z\partial_z N = Az -  \frac z2 \int_z^1  \frac{\psi(s)}{s^2}ds
 +  \frac 1{2z}  \int_0^z  {\psi(s)} ds =   O(z^{\alpha+\rho})
  \,.
$$
If $\alpha + \rho>1$ we can write instead
\beal{16X15.2}
 N  & = &  Az+  \frac z2 \int_0^z  \frac{\psi(s)}{s^2}ds
 -  \frac 1{2z}  \int_0^z  {\psi(s)} ds = Az+  O(z^{\alpha+\rho})
  \,,
\\
 z\partial_z N  &= & Az+  \frac z2 \int_0^z  \frac{\psi(s)}{s^2}ds
 +  \frac 1{2z}  \int_0^z  {\psi(s)} ds =  Az+ O(z^{\alpha+\rho})
  \,.
\eeal{4X15.1}

The case $\rho=1$ requires special consideration. (Note that we can always increase $\rho$ slightly if necessary if $\hat \rho$ is \emph{not} attained, so $\rho=1$ needs to be considered in our argument only if $\hat \rho=1$ and is attained.) In this case, comparing \eq{16X15.1} with  \eq{16X15.2} we see that
$ f (z)  = A  z$, so that $A$ is constant. We then have, in local coordinates, using \eq{christ_zg1}-\eq{christ_zg2},
$$
 \nabla_a \nabla_b N = \partial_a \partial_b N - \Gamma^k_{ab}\partial_k N = \partial_a \partial_b N - \frac A z \ringh_{ab} + O(z^{\alpha -1})
 \,,
$$
while the right-hand side of \eq{2X15.5}
 equals $ \frac A z \ringh_{ab}  + O(z^{\alpha -1})$. Hence the limit $\lim_{z\to 0} z \partial_a \partial_b N$ exists and is equal to
 $2 A \ringh_{ab}$. Using \eq{16X15.1}, for all nearby $\theta$ and $\theta_1$ we have
\beaa
 0 &= &  \lim_{z\to0}z\partial_aN(z,\theta_1)
 -
 \lim_{z\to0}z\partial_aN(z,\theta)
\\
  & = &
  \lim_{z\to0} z(\theta_1^b-\theta^b) \int_{0}^1 \partial_a\partial_b   N(z,(t \theta_1+(1-t)\theta ))
  \, dt
\\
  &  = &
    2  A  (\theta_1^b-\theta^b)
    \int_{0}^1\ringh_{ab}  (t \theta_1+(1-t)\theta ) \, dt
      \,.
\eeaa
Since this holds for all $\theta$ and $\theta_1$, and $\ringh$ is non-degenerate, we see that   $A=0$.

Taking into account that $A$ must vanish as well if $\rho>1$, we have  shown that
\bel{4X15.2}
 |N| + |\nabla  N|_g =
  \left\{
    \begin{array}{ll}
      O(z^{\alpha +\rho}), & \hbox{$\alpha + \rho < 1$ or $\rho\ge 1$;} \\
      O(z), & \hbox{$\alpha + \rho > 1$ and $\rho<1$.}
    \end{array}
  \right.
\ee
When $\rho<1$, we can decrease $\alpha$ if necessary to have  $\alpha+\rho<1$, which we assume from now on.
We conclude that
\bel{4X15.3}
 |N| + |\nabla  N|_g =
      O(z^{\alpha +\rho})
\,.
\ee
%
\Eq{2X15.5} gives
\bel{4X15.7}
 |\nabla \nabla  N|_g =
      O(z^{\alpha +\rho})
\,.
\ee

The right-hand side of \eq{2X15.8}, as well as the $z\partial_b$-derivatives thereof of order one and two, are now $O(z^{\alpha+\rho})$. Integrating in $z$ the resulting equations, one finds
\bel{1XI15.5}
 Y_z(z,\theta) = \left\{
                   \begin{array}{ll}
                   \displaystyle
   z^{-1}Y_z(1,\theta) - z^{-1} \int_z^1 \frac{\lambda(s,\theta)}s\, ds, & \hbox{$\alpha+\rho<0$;} \\
                   \displaystyle
                      z^{-1}Y_z(1,\theta) - z^{-1} \int_0^1 \frac{\lambda(s,\theta)}s\, + z^{-1} \int_0^z \frac{\lambda(s,\theta)}s\, ds, & \hbox{$\alpha+\rho>0$.}
                   \end{array}
                 \right.
\ee
In either case there exists a function $f(\theta)$
such that
%
\bel{2X15.10asdf}
 Y_z = \frac{ f(\theta)}z + O(z^{\alpha+\rho-1})
 \,,
\ee
with a similar behavior for $z\partial_b$-derivatives of order one and two
of $Y_z$.
The function $f$ is clearly as differentiable as $Y_z$ if $\alpha+\rho<0$,
and so is the error term in this case. On the other hand, $f$ is in $C^{k+\sigma} $
for any $0\le k+\sigma < \alpha+\rho $  when $\alpha+\rho>0$ (as follows from elementary estimates and the interpolation inequality in the $\theta$-variables,
$$
 \|\lambda(s,\cdot)\|_{C^{0,\sigma}} \le C \|\lambda(s,\cdot)\|_{C^{0 }}^{1-\sigma} \|\lambda(s, \cdot) \|_{C^{1}}^\sigma
\,,
$$
applied to the integrand of the middle term at the right-hand side of the second line of \eq{2X15.10asdf})), but more regularity is not clear. For this reason it is convenient to replace $f(\theta)$ by a function  $f(z,\theta) $ as in~\cite[Lemma~3.3.1 and Corollary 3.3.2]{AndChDiss} which is smooth for $z>0$, which approaches  $f(\theta)$ as $O(z^{\alpha/2+\rho})$,   with $f(z,\theta)\equiv 0$ if $f(\theta)\equiv 0$, and with derivatives behaving in an obvious  weighted way, in particular
\bel{3XI15.1}
 \partial_af(z,\theta)= O(z^{\min(0,\alpha/2+\rho-1)})\;,\;\;\partial_a \partial_b   f(z,\theta) = O(z^{\min(0,\alpha/2+\rho-2)})
 \,.
\ee
To avoid annoying factors of two, from now on the symbol $\alpha$ stands for one-half of the previously used value of $\alpha$.
This leads to
%
\bel{2X15.10z}
 Y_z = \frac{ f(z, \theta)}z + \frac{ f(\theta)-f(z,\theta)}z + O(z^{\alpha+\rho-1}) =  \frac{ f(z, \theta)}z + O(z^{\alpha+\rho-1})
 \,.
\ee
We conclude that
\bel{2X15.10}
 Y_z = \frac{ f(z,\theta)}z + O(z^{\alpha+\rho-1})
 \,,
  \quad
 z\partial_i Y_z = z\partial_i\left(\frac{ f(z,\theta)}z \right) + O(z^{\alpha+\rho-1})
 \,.
\ee
Differentiating \eq{2X15.8} leads further to
\bel{2X15.10+}
 z^2\partial_i \partial_z Y_z =  z^2 \partial_i\partial_z \left(\frac{ f(z,\theta)}z \right) + O(z^{\alpha+\rho-1})
 \,.
\ee
Inserting \eq{2X15.10} into \eq{2X15.9} and its $\partial_b$--derivative, and solving the ODE  we obtain
\bel{2X15.11}
 Y_a = - \frac 12 \partial_a f (z,\theta) + O(z^{\alpha+\rho-1} )
  \,,
 \quad
 z\partial_i Y_a = - \frac z2 \partial_i\partial_a f (z,\theta) +
    O(z^{\alpha+\rho-1})
 \,.
\ee
The $ab$-components of \eq{2X15.4} read
\be
 \label{2X15.4a}
 \nabla_{(a}Y_{b)}
 = K_{ab}N = O_g(z^{\alpha+\rho})
 \,.
\ee
The left-hand side equals
\bea
 \label{2X15.4b}
 \lefteqn{
 \nabla_{(a}Y_{b)}
  = \partial_{(a} Y_{b)} - \Gamma^k_{ab} Y_k
 }
&&
\\
  &   & =
    - \frac 1{2} \zmcD_a \zmcD_b f - \frac{f }{z^2} \ringh_{ab} + O(z^{\alpha+\rho-2} )+ O(z^{\alpha-1} )f + O(z^{\alpha} )\partial_c f
 \,,
 \nn
\eea
where $\zmcD_a$ denotes the covariant derivative of the metric $\ringh$. Comparing with \eq{3XI15.1} and \eq{2X15.4a}   we conclude that $f\equiv 0$ if $\alpha+\rho>0$, so that in all cases we have
\bel{2X15.10c}
 Y_k =  O(z^{\alpha+\rho-1})
 \,,
  \quad
 z\partial_i Y_k =   O(z^{\alpha+\rho-1})
 \,,
\ee
equivalently
\bel{2X15.10d}
 | Y|_g + |\nabla Y|_g =  O(z^{\alpha+\rho })
 \,.
\ee
Thus
\bel{2X15.12}
  |N|+|Y|_g+|\nabla N|_g+|\nabla Y|_g = O(z^{\alpha+\rho })
 \,.
\ee
Since $\alpha+\rho>\hat \rho$, this contradicts the definition of $\hat \rho$ when $\hat \rho <\infty$.

We conclude that $\hat \rho = \infty$.
Thus $N$ and $Y$ decay arbitrary fast at the conformal boundary.

To finish the proof, let
$$
 f:= N^2 +|Y|_g^2+|\nabla N|_g^2+|\nabla Y|_g ^2
 \,.
$$
\Eq{2X15.5} and \eq{2X15.16} imply that there exists a constant $C>0$ such that
$$
 z \partial_z f \le C f
 \,.
$$
Equivalently,
$$
 \partial_z ( z^{-C} f ) \le 0
 \,.
$$
So  $ z^{-C} f $ is decreasing, non-negative, and tends to zero as $f$ approaches zero. We conclude that $f\equiv 0$ and the result is established.
\qedskip

\section{Asymptotic behavior near the corner}
 \label{A12VII15.1}

We consider a manifold $\Omega$ with a corner
of the form
$$
 \Omega= \{x\ge 0\,,\ z\ge 0\}\times N
 \,,
$$
where $N$ is a compact manifold, with a smooth metric
\bel{15IX15.2}
 g= z^{-2}\underbrace{(dz^2 +dx^2+h)}_{=:\tilde g}
\ee
where $h$ is a smooth Riemannian metric on $N$. In our applications in the main body of this work neither the manifold nor $g$ are of this form, but $g$ is suitably equivalent to \eq{15IX15.2} in local coordinates near the corner, so that the estimates below  apply.

The aim of this appendix is to prove the following sharp estimate:

\begin{Proposition}
  \label{P15IX15.1}
Let $k>n/2$, let $\alpha$ be any number in $(0,1)$ when $\frac n2\in\Z$, and $\alpha=
 \frac 12 $  if $\frac n2\notin\Z$. Define
$$
  \rho=\sqrt{x^2+z^2}\,, \quad
\phi=\frac x\rho\,,\;\;\;\psi= x^az^b\rho^c\,,\;\;\; \varphi=x^{a+\frac n2}z^{b}\rho^{c-\frac n 2}
 \,.
$$
Then there exists a constant $C$ such that
\bel{17IX15.6}
\|u\|_{C^{k-n/2-1+\alpha}_{\phi,\varphi}(\Omega)}\leq
 C \|u\|_{H^k_{\phi,\psi}(\Omega)}
 \,.
\ee
Moreover it holds
\bel{15IX15.1}
  u\in \nozHkabc
\,,\ k>n/2
 \quad
 \Longrightarrow
 \quad
 u = o(x^{-a-n/2} z^{-b  } \rho^{ -c+n/2})
 \,,
\ee
for small $\rho$, similarly for weighted derivatives of $u$ of order strictly smaller than $k-n/2$.
\end{Proposition}

\begin{remark}{\rm
   \label{R15IX15.1}
The estimate \eq{15IX15.1} is standard away from the corner $x=z=0$ (thus, for $x>\epsilon>0$ or $z>\epsilon>0$ or both); cf., e.g.,~\cite[Theorem~2.3]{andersson:elliptic}.
\myqed
}\end{remark}

\proof
Let $(x,y,z)$ be a natural coordinate system near the corner, where $y=(y^A)$ is a local coordinate system on $N$. Without loss of generality we can assume that the coordinates range over
$$
 (x,y,z) \in (0,4) \times(-2,2)^{n-2} \times(0,4)
  \,.
$$
By Remark~\ref{R15IX15.1}, it suffices to prove \eq{15IX15.1} near points $  (x_0,y_0,z_0)$ such that
$$
 (x_0,y_0,z_0)\in  (0,1) \times(-1,1)^{n-2} \times(0,1)
  = :{\mathcal C}
  \,.
$$

For $0<\epsilon  <1$  let
$$
\begin{array}{llll}\varphiepsiloneta :&{\mathcal C}&\longrightarrow& \R^n\,,\\
&(\hat x,\hat y,\hat z)&\mapsto&(x,y,z):=(x_0+\epsilon \hat x,y_0+\epsilon \hat y,z_0+\etaepsilon \hat z)
 \,.\end{array}
$$
Set
$$
v=u\circ \varphiepsiloneta  \,.
$$
Then
$$
\partial_{\hat z}v=\epsilon(\partial_zu)\circ\varphiepsiloneta  \,,
 \quad
\partial_{\hat x}v=\etaepsilon (\partial_xu)\circ\varphiepsiloneta  \,,
    \quad
\partial_{\hat y}v=\epsilon (\partial_yu)\circ\varphiepsiloneta  \,.
$$

We consider the  norms $
\|u\|_{\mathring H^k_{\phi,\psi}(\varphiepsiloneta (\mathcal C))}
$, which are all finite.  In fact it  follows from the dominated convergence theorem that
\bel{16IX15.3}
\|u\|_{\mathring H^k_{\phi,\psi}(\varphiepsiloneta (\mathcal C))}\to 0
\ee
when $\epsilon $  goes to zero.

Suppose, first, that $x_0>2z_0$. Choose $\epsilon   = z_0/3$.
On $\varphiepsiloneta (\mathcal C)$  we have
$$
x\approx x_0\,,\;\;z\approx z_0\approx \epsilon\,,\;\;\rho\approx \rho_0:=\sqrt{x_0^2+z_0^2}\approx x_0
\,,
$$
where ``$s \approx t$'' means that there exists a constant $C>0$ such that $C^{-1} s \le t \le C s$.

In local coordinates, the Riemannian measure $d\mu_g$ associated to $g$ and appearing in the integrals defining the norm $
\|u\|_{\mathring H^k_{\phi,\psi}(\varphiepsiloneta (\mathcal C))}
$ is equivalent to
$$
d\mu:=z^{-n}dx\,dy^{n-2} dz\approx z_0^{-n }\etaepsilon^n  d\hat x\,d\hat y^{n-2}  d\hat z
 \,.
$$

It follows from the definition that for $0\le |\beta| \le k$ we have
$$
 ({xz}\rho^{-1})^{|\beta|}\partial ^\beta u\in L^2(\Omega,x^{2a}z^{2b-n}\rho^{2c}dx\,d^{n-2}y\,dz)
$$
or, equivalently, for all $0\le |\beta|=i\le k$,
\beaa
&
\partial^\beta u\in L^2(\Omega,x^{2a+2i}z^{2b-n+2i}\rho^{2c-2i}dx\,d^{n-2}y\,dz)
 \,.
  &
\eeaa
This implies that  for all $ 0\le i_x+i_z+|\alpha|=i\le k $
we have
\bel{16IX15.1}
 \phantom{xxx}
\partial_{\hat x}^{i_x} \partial_{\hat z}^{i_z} \partial_{\hat y}^{\alpha} v
 \in
  L^2\bigg({\mathcal C},
   \underbrace{
    x_0^{2a+2i}z_0^{2b-n+2i}\rho_0^{2c-2i} \epsilon^{-2i+n}
       }_{
         \approx  x_0^{2a }z_0^{2b  }\rho_0^{2c }
          \approx  x_0^{2a +n}z_0^{2b  }\rho_0^{2c -n}
        }
 d\hat x\,d^{n-2}\hat y\,d\hat z\bigg)
\,,
\ee
with the norm in the space there going to zero as $\epsilon$ goes  to zero by \eq{16IX15.3}.

By the Sobolev embedding on $\mathcal C$,
$v$ is pointwise bounded  by its $H^k(\mathcal C)$ norm. Hence
$$
v= o(
     x_0^{-a -n/2}z_0^{-b  }\rho_0^{-c+n/2 } )\,.
$$
Then  on $\varphiepsiloneta (\mathcal C)$ we obtain
\bel{17IX15.5}
u= o(
     x ^{-a -n/2}z ^{-b  }\rho ^{-c+n/2 } ) \,.
\ee

Suppose, next, that $x_0/2 \le  z_0\le 2x_0$. Choose   $\epsilon = x_0/6$.
On $\varphiepsiloneta (\mathcal C)$  we have
$$
x\approx x_0\approx \epsilon\,,\;\;z\approx z_0\approx \epsilon \,,\;\;\rho\approx \rho_0 \approx \epsilon
\,,
$$
\Eq{16IX15.1} becomes now
\bel{16IX15.1+}
 \phantom{xxx}
\partial_{\hat x}^{i_x} \partial_{\hat z}^{i_z} \partial_{\hat y}^{\alpha} v
 \in
  L^2\bigg({\mathcal C},
   \underbrace{
    x_0^{2a+2i}z_0^{2b-n+2i}\rho_0^{2c-2i} \epsilon^{-2i+n}
       }_{
         \approx  \epsilon^{2a +2b +2c }
          \approx  x_0^{2a +n}z_0^{2b  }\rho_0^{2c -n}
        }
 d\hat x\,d^{n-2}\hat y\,d\hat z\bigg)
\,,
\ee
which leads again to \eq{17IX15.5}.

Suppose, finally, that $z_0>2x_0$. Choose $\epsilon   = x_0/3$.
On $\varphiepsiloneta (\mathcal C)$  we have
$$
x\approx x_0\approx \epsilon\,,\;\;z\approx z_0\,,\;\;\rho\approx \rho_0 \approx z_0
\,,
$$
Obvious modifications of the calculations above lead again to \eq{17IX15.5}, and
\eqref{15IX15.1} is established.

In a similar way, we can use the Sobolev embedding
$$
\|v\|_{C^{k-n/2-1,\alpha}(\mathcal C)}\leq C \|v\|_{H^k(\mathcal C)}
 \,,
 \quad k>\frac n2
 \,,
$$
to obtain \eq{17IX15.6}.
\qed

\begin{remark}{\rm
  One can find a number $N<\infty$ and a covering of $\Omega$ by cubes ${\mathcal C}_i:=\varphi_{\epsilon_i}({\mathcal C})$ as in the proof of Proposition~\ref{P15IX15.1}, with centers $(x_i,y_i,z_i)$, so that every point in $\Omega$ is included in at most $N$ such cubes. It then follows from the proof above that
  %
\bel{17IX15.7}
   \sum_i\|u\|_{C^{k-n/2-1+\alpha}_{\phi,\varphi}({\mathcal C}_i)}
    \leq C \|u\|_{H^k_{\phi,\psi}(\Omega)}
 \,,
\ee
for some constant $C$.
}
\myqed
\end{remark}

\section{Differentiability of the constraint map}
 \label{A21XI15.1}

We sketch the argument   justifying that the map
\bel{16VI16.4}
 (\delta K,\delta g)\mapsto \sources(K+\delta K,g+\delta g)-\sources(K,g),
\ee
where the constraint map $\sources$ has been defined in (\ref{defJrho}),
is well defined and differentiable near zero on the spaces we work with.
This is well described for instance in the proof of Corollary~3.2 of \cite{bartnik:phase}.

First, when $\delta g$ is small in $\phi^2\mathring  H^{k+2}_{\phi,\psi^{-1}}$, with $a$ large, $b\geq 0$ and $k>n/2$ then (see Proposition~\ref{P15IX15.1}) $g+\delta g$
remains positive definite.

Next, we  have to verify that the map \eq{16VI16.4} is locally bounded  near zero from
$$
 X:=\psi^2(\phi \mathring H^{k+2}_{\phi,\psi}, \phi^2 \mathring H^{k+2}_{\phi,\psi})
=(\phi \mathring H^{k+2}_{\phi,\psi^{-1}}, \phi^2 \mathring H^{k+2}_{\phi,\psi^{-1}})
$$
to
$$
 Y:=\psi^2( \mathring H^{k+1}_{\phi,\psi},  \mathring H^{k}_{\phi,\psi})=
( \mathring H^{k+1}_{\phi,\psi^{-1}},  \mathring H^{k}_{\phi,\psi^{-1}})
 \,.
 $$
This is the case for large $a$, $b\geq 0$ and $k>n/2$ so  that,
see Proposition~\ref{P15IX15.1}, the space $H^k_{\phi,\psi^{-1}}$ is an algebra.
At this stage it remains to note that all the estimates can be choosen uniform
for initial data close to $(K,g)$ in $W^{k+3,\infty}_\phi\times W^{k+4,\infty}_\phi$,
in particular for initial data close to $(K,g)$
in $C^{k+3}(\overline \Omega)\times C^{k+4}(\overline \Omega)$.
Here
$$
W^{k,\infty}_{\phi}:=\{u\in W^{k,\infty}_{\loc} \mbox{ such that for $0\le i\le k$ we have\ }
\phi^i|\nabla^{(i)}u|_g\in L^{\infty}\}\,,
$$
with the obvious norm, and with $\nabla^{(i)}u$ --- the tensor of
(possibly distributional) $i$-th covariant derivatives of $u$.

\bigskip

\noindent
{\sc Acknowledgements:} {Work supported in parts by the Agence Nationale de la Recherche
 through grants ANR  SIMI-1-003-01 and  ANR-10-BLAN 0105.}

\bibliographystyle{amsplain}
\bibliography{../references/hip_bib,%
../references/reffile,%
../references/newbiblio,%
../references/newbiblio2,%
../references/chrusciel,%
../references/bibl,%
../references/howard,%
../references/bartnik,%
../references/myGR,%
../references/newbib,%
../references/Energy,%
../references/dp-BAMS,%
../references/prop2,%
../references/besse2,%
../references/netbiblio,%
../references/PDE}

\end{document}